\documentclass[a4paper,11pt,reqno]{amsart}
\usepackage{amsmath}
\usepackage{relsize}
\usepackage[noadjust]{cite}
\usepackage{color}
\usepackage[dvipsnames]{xcolor}
\usepackage{mathtools}
\usepackage[normalem]{ulem}
\usepackage{comment}
\newcommand\redout{\bgroup\markoverwith
{\textcolor{red}{\rule[.5ex]{2pt}{0.4pt}}}\ULon}

\usepackage{hyperref, enumitem}
\hypersetup{
  colorlinks   = true, 
  urlcolor     = blue, 
  linkcolor    = Purple, 
  citecolor   = red 
}
\usepackage{amsmath,amsthm,amssymb}

\usepackage[margin=2.5cm]{geometry}
\usepackage{caption} 

\usepackage{tikz-cd}
\usetikzlibrary{calc,fit,matrix,arrows,automata,positioning}
\usetikzlibrary{decorations.pathreplacing,angles,quotes}
\usepackage{stmaryrd}

\usetikzlibrary{fit}
\tikzset{%
  highlight/.style={rectangle,rounded corners,fill=red!15,draw,
    fill opacity=0.5,thick,inner sep=0pt}
}

\usepackage{array}
\newcolumntype{x}[1]{>{\centering\arraybackslash\hspace{0pt}}p{#1}}

\theoremstyle{definition}
\newtheorem{theorem}{Theorem}[section]
\newtheorem{definition}[theorem]{{{Definition}}}
\newtheorem{example}[theorem]{{{Example}}}
\newtheorem{notation}[theorem]{{{Notation}}}
\newtheorem{remark}[theorem]{{{Remark}}}

\newtheorem{corollary}[theorem]{{{Corollary}}}
\newtheorem{proposition}[theorem]{{{Proposition}}}
\newtheorem{lemma}[theorem]{{{Lemma}}}

\newcommand{\C}{\mathcal C}

\newcommand{\mM}{\mathcal{M}}
\newcommand{\mL}{\mathcal{L}}
\newcommand{\mU}{\mathcal{U}}
\newcommand{\mF}{\mathcal{F}}

\newcommand{\mI}{\mathcal{I}}
\newcommand{\mZ}{\mathcal{Z}}
\newcommand{\mS}{\mathcal{S}}
\newcommand{\mT}{\mathcal{T}}
\newcommand{\F}{\mathbb F}
\newcommand{\la}{\langle}
\newcommand{\ra}{\rangle}

\newcommand{\Fq}{\F_q}

\DeclareMathOperator{\GL}{GL}
\DeclareMathOperator{\supp}{supp}

\DeclareMathOperator{\rk}{rk}

\DeclareMathOperator{\dd}{d}

\DeclareMathOperator{\rr}{r}

\DeclareMathOperator{\wt}{wt}

\DeclareMathOperator{\rowsp}{rowsp}

\newcommand{\dr}{\dd_{\rr}}

\newcommand{\Fmkd}{[n,k,d]_{q^m/q}}
\newcommand{\Fmk}{[n,k]_{q^m/q}}
\newcommand{\Fm}{\F_{q^m}}

\usepackage{scrextend}
\deffootnote{0em}{0em}{\thefootnotemark\quad}

\title{Representability of the direct sum of uniform $q$-matroids}
\usepackage[foot]{amsaddr}
\author[G. N. Alfarano]{Gianira N. Alfarano$^{1,2,3}$}
\address{$^1$\textnormal{Université de Rennes, IRMAR - UMR 6625, Rennes Cedex, France.}}
\address{$^2$\textnormal{School of Mathematics and Statistics, University College Dublin, Science Centre Belfield Dublin 4, Ireland.}}
\address{$^3$\textnormal{Department of Mathematics and Data Science, Vrije Universiteit Brussel, B-1050 Brussels, Belgium.}}
 \email{gianira-nicoletta.alfarano@univ-rennes.fr}
 \author[R. Jurrius]{Relinde Jurrius$^4$}
 \address{$^4$\textnormal{Faculty of Military Sciences, Netherlands Defence Academy, the Netherlands.}}
 \email{rpmj.jurrius@mindef.nl}
\author[A. Neri]{Alessandro Neri$^{5,6}$}
\address{$^5$\textnormal{Department of Mathematics and Applications ``R. Caccioppoli'', University of Naples Federico II, Via Cintia, Monte S. Angelo, 80126 Napoli, Italy.}}
\address{$^6$\textnormal{Department of Mathematics: Analysis, Logic and Discrete Mathematics, Ghent University, Krijgslaan 281, 9000 Gent, Belgium.}}
 \email{alessandro.neri@unina.it}
\author[F. Zullo]{Ferdinando Zullo$^7$}
\address{$^7$\textnormal{Department of Mathematics and Physics, University of Campania ``Luigi Vanvitelli'', Viale Lincoln 5, 81100 Caserta, Italy.}}
 \email{ferdinando.zullo@unicampania.it}

\begin{document}

\begin{abstract}
  There are many similarities between the theories of matroids and $q$-matroids. However, when dealing with the direct sum of $q$-matroids many differences arise. Most notably, it has recently been shown that the direct sum of representable $q$-matroids is not necessarily representable. In this work, we focus on the direct sum of uniform $q$-matroids. Using algebraic and geometric tools, together with the notion of cyclic flats of $q$-matroids, we show that this is always representable, by providing a representation over a sufficiently large field.
\end{abstract}

\maketitle


\section*{Introduction}

The concept of $q$-matroid traces back to Crapo's PhD thesis \cite{crapo1964theory}. More recently, $q$-matroids have been reintroduced in \cite{jurrius2018defining} and they gained a lot of attention because of their relation to rank-metric codes; see for instance \cite{gorla2019rank, shiromoto2019codes, ghorpade2020polymatroid, byrne2022constructions, byrne2021weighted, gluesing2021q}.  

A full exposition of \emph{cryptomorphisms}, i.e. many equivalent ways to describe a $q$-matroid axiomatically, has been given in \cite{byrne2022constructions}, in terms of rank function, independent spaces, flats, circuits, bases, spanning spaces, closure function, hyperplanes, open spaces etc. Later, another cryptomorphism based on cyclic flats has been derived in \cite{alfarano2022cyclic}. In general, these are not straightforward $q$-analogues of the traditional matroids' ones. 

While there are many similarities between matroids and $q$-matroids, there are substantial differences when dealing with the direct sum. Recently, the direct sum of $q$-matroids has been introduced in terms of the rank function; see \cite{ceria2021direct}. In the same paper, some fundamental properties have been established. Moreover in \cite{gluesing2023coproducts}, it has been shown that the direct sum of $q$-matroids is a coproduct in the category of $q$-matroids with linear weak maps as morphisms. This also implies that the direct sum of $q$-matroids on ground spaces $E_1$ and $E_2$ is the unique $q$-matroid with the most independent spaces among all $q$-matroids on $E_1 \oplus E_2$ whose restrictions to $E_i$ are isomorphic to the given $q$-matroids. In \cite{gluesing2023decompositions}, it is shown that the cyclic flats are one of the few objects that behave well with the direct sum of $q$-matroids. In \cite{gluesing2022representability}, the first steps towards the study of the representation of the direct sum of $q$-matroids have been taken. For a $q$-matroid with ground space $\F_q^n$, representability can be defined over any field extension of $\F_q$; see Section \ref{subsec:preliminaries_q-mat}. One big difference with the representation of matroids is that in $q$-matroids the characteristic of the field is fixed, and the only free parameter is the degree of the field extension. Moreover, while the direct sum of matroids is always representable, it has been shown that the direct sum of two representable $q$-matroids may not be representable over any field, or it may require a sufficiently large field; see~\cite{gluesing2022representability}. Representability of $q$-matroids has been geometrically described in \cite{alfarano2022cyclic}, where it is illustrated how to equivalently define the rank function of a representable $q$-matroid in terms of a $q$-system, a notion introduced in \cite{sheekey2019scatterd, randrianarisoa2020geometric}.

\medskip 

\paragraph{\textbf{Our contribution:}} In this paper we study the direct sum of $t$ many uniform $q$-matroids. These are special $q$-matroids, since they are representable as  \emph{maximum rank distance (MRD) codes} or (equivalently) as \emph{scattered subspaces}.

Our first main result is the geometric characterization of the representation of the direct sum of uniform $q$-matroids. We prove that a $q$-system is a representation of such direct sum if and only if it satisfies some evasiveness properties; see Theorem \ref{thm:charact_independence}. 
The crucial objects towards our characterization are the cyclic flats, which for uniform $q$-matroids are just the $0$-subspace and the whole ground space. 

The second main result is the construction of a $q$-system (over a sufficiently large field) that satisfies the previously mentioned evasiveness properties; see Theorem \ref{thm:direct_sum_constr_general}. This, in particular, shows that the direct sum of uniform $q$-matroids is always representable; see Theorem \ref{thm:representability_direct_sum}. However, it remains  an open question which is the smallest extension field over which we can find a representation of this direct sum. 

In the last part of the paper we study the extension fields over which the direct sum of two uniform $q$-matroids of rank $1$ is representable. This is a special case, since the associated system has been already investigated in terms
of linear sets with complementary weights; see~\cite{napolitano2022linear}. Thanks to this point of view, we can provide more precise results on the field size needed for the representation of two uniform $q$-matroids of rank $1$; see Theorem \ref{thm:summary}.

\medskip

\paragraph{\textbf{Outline of the paper:}} The rest of the paper is organized as follows. In Section \ref{sec:preliminaries}, we provide the needed background. In Section \ref{sec:direct_sum_uniform}, we characterize the direct sum of $t$ many uniform $q$-matroids in geometric terms. In Section \ref{sec:construction}, we exhibit a representation of such direct sum. In Section \ref{sec:point_point}, we restrict ourselves to study the direct sum of two uniform $q$-matroids of rank $1$. We draw some conclusions and open problems in Section \ref{sec:conclusion}.



\section{Preliminaries}\label{sec:preliminaries}

In this section we provide the necessary background for the rest of the paper. 

\subsection{Rank-metric codes and \emph{q}-systems}

Rank-metric codes have been introduced originally by Delsarte in~\cite{de78} and by Gabidulin in \cite{ga85a}. Let $\F_q$ be the finite field with $q$ elements and let $m,n$ be positive integers with $m\geq n$. We endow the vector space $\F_{q^m}^n$ with the \textbf{rank metric}, defined for every $x=(x_1,\ldots,x_n)$, $y=(y_1,\ldots,y_n)\in\F_{q^m}^n$ as 
\[\dr(x,y) = \dim_{\F_q}\langle x_1-y_1,\ldots, x_n-y_n\rangle_{\F_q}.\]
 An $[n,k]_{q^m/q}$ \textbf{rank-metric code} $\C$ is a $k$-dimensional $\F_{q^m}$-linear subspace of $\F_{q^m}^n$. The \textbf{minimum rank distance} of $\C$ is defined as
 \begin{align*}
     \dr(\C):= &\min\{ \dr(x,y) \mid x,y \in \C,\,\, x\neq y \}.
 \end{align*}

If $d=\dr(\C)$ is known, we write that $\C$ is an $[n,k,d]_{q^m/q}$  code.
The parameters $n,m,k,d$ of an $[n,k,d]_{q^m/q}$  rank-metric code satisfy a Singleton-like bound \cite{de78}, that reads~as
\begin{equation}\label{eq:singleton} k \leq n-d+1. \end{equation}
When equality holds, we say that $\C$ is a \textbf{maximum rank distance} code or \textbf{MRD} for short.
It has been shown that MRD codes exist if and only if $n\le m$, for any choice of $d\le n$, over any finite field, and over any field admitting a cyclic Galois extension of degree degree $m$; see~\cite{de78,ga85a,guralnick1994invertible,roth1996tensor}. Note that we are defining MRD codes starting from the Singleton-like bound in \eqref{eq:singleton}. However, sometimes in the literature one also admits the transpose Singleton-like bound, for which MRD codes may exist also when $n>m$.

For a vector 
$v \in \F_{q^m}^n$ and an ordered basis $\Gamma=\{\gamma_1,\ldots,\gamma_m\}$ of the field extension
$\F_{q^m}/\F_q$, let $\Gamma(v) \in \F_q^{n \times m}$ be the matrix defined by
$$v_i= \sum_{j=1}^m \Gamma(v)_{ij} \gamma_j.$$
The \textbf{support} of $v$ is the column space of $\Gamma(v)$ for any basis $\Gamma$; we denote it by $\supp(v)$.
The \textbf{rank-weight} of $v\in\F_{q^m}^n$ is the quantity $\rk(v)=\dim_{\F_q}(\supp(v))$.
A \textbf{generator matrix} for an $[n,k,d]_{q^m/q}$  code $\C$ is a full-rank matrix $G\in\F_{q^m}^{k\times n}$ such that $\C=\rowsp(G)$. If the columns of one (and hence any) generator matrix of $\C$ are $\F_q$-linearly independent, then $\C$ is said to be \textbf{nondegenerate}. 
Two $[n,k,d]_{q^m/q}$ rank-metric codes $\C$ and $\C'$ are {\textbf{equivalent}} if and only if
there exists $A \in \mathrm{GL}(n,\F_q)$ such that
$\C'=\C A=\{vA : v \in \C\}$.

There is a geometric interpretation of $[n,k,d]_{q^m/q}$ rank-metric codes as $q$-\emph{systems}. 
 \begin{definition}
An $\Fmkd$ {(or $\Fmk$)} \textbf{system} $\mS$ is an $n$-dimensional $\F_q$-subspace of $\F_{q^m}^k$, such that
$ \langle \mS \rangle_{\F_{q^m}}=\F_{q^m}^k$ and
$$ d=n-\max\left\{\dim_{\F_q}(\mS\cap H) \mid H \textnormal{ is an } \F_{q^m}\textnormal{-hyperplane of }\F_{q^m}^k\right\}.$$
 {The \textbf{full $q$-system} of $\F_{q^m}^k$ is the $[mk,k]_{q^m/q}$ system $\mS=\F_{q^m}^k$.}
Moreover, two $[n,k,d]_{q^m/q}$ systems $\mS$ and $\mS'$ are \textbf{equivalent} if there exists an $\F_{q^m}$-isomorphism $\varphi\in\GL(k,\F_{q^m})$ such that
$$ \varphi(\mS) = \mS'.$$
If the parameters are not relevant or clear from the context, we simply say that $\mS$ is a $q$-\textbf{system}.
\end{definition}

Although the term ``$q$-system'' has been introduced only a few years ago, the theory of these objects is strictly related to the one of \textbf{linear sets}. In some sense, $q$-systems and linear sets can be thought of as the same concept. Thus, we will report a few results and properties originally studied in the theory of linear sets, using the language of $q$-systems. The interested reader is referred to \cite{polverino2010linear} for an in-depth treatment of linear sets. 

\begin{definition}
Let $\mS$ be an $\Fmk$ system. For each $\Fq$-subspace $V$ of  $\Fm^k$, we define the \textbf{weight} of $V$ in $\mS$ as the integer
$$ \wt_{\mS}(V):=\dim_{\F_q}(\mS\cap V).$$
\end{definition}

In \cite{sheekey2019scatterd, randrianarisoa2020geometric} it has been proved that there is a one-to-one  correspondence  between  equivalence  classes of nondegenerate $\Fmkd$ codes and equivalence classes of $\Fmkd$ systems and the correspondence can be explained as follows.
Let $\C$ be an $[n,k,d]_{q^m/q}$ code and $G$ be a generator matrix for $\C$. Define $\mS$ to be the $\F_q$-span of the columns of $G$;  {it turns out to be an $[n,k,d]_{q^m/q}$ system}. In this case $\mS$ is also called a $q$-\textbf{system associated with} $\C$. Vice versa, given an $[n,k,d]_{q^m/q}$ system $\mS$, define $G$ to be the matrix whose columns are an $\F_q$-basis of $\mS$ and let $\C$ be the code generated by $G$;  {it turns out to be an $[n,k,d]_{q^m/q}$-code}. $\C$ is also called a \textbf{code associated with} $\mS$. For more details on this correspondence, see also \cite{alfarano2022linear}.

Finally recall the following result which provides a natural description of the supports of codewords of an $[n,k]_{q^m/q}$ code $\C$ in terms of a $q$-system $\mS$ associated to $\C$.
\begin{theorem}[{\cite[Theorem 3.1]{neri2021geometry}}]\label{thm:psiG}
Let $\C$ be an $[n,k]_{q^m/q}$ non-degenerate rank-metric code and let $\mS$ be the $\F_q$-span of the columns of a generator matrix $G$. Consider the isomorphism
\begin{align}\label{eq:psiG}
    \psi_G: \F_q^n \to \mS, \quad  v \mapsto vG^\top.
\end{align}
For every $u\in\F_{q^m}^k$ we have that 
$$\psi_G^{-1}(\mS \cap \langle u \rangle ^\perp)=\supp(uG)^\perp.$$
\end{theorem}

\begin{definition} \label{def:ahevasive}
Let $\mathcal{A}$ be a collection of $\F_q$-subspaces of $\F_{q^m}^k$, let $\mS$ be an $\Fmk$ system and let $h$ be a positive integer. We say that
$\mS$ is $(\mathcal{A},h)$-\textbf{evasive} if
\[ \wt_{\mS}(A)\leq h, \quad \mbox{for all $A \in \mathcal{A}$}. \]
\end{definition}

Remark that since every $\F_{q^m}$-linear subspace is also $\F_q$-linear, we can apply the above definition to $\F_{q^m}$-linear subspaces as well.

Denote by $\Lambda_h$ the set of all the $h$-dimensional $\F_{q^m}$-subspaces of $\F_{q^m}^k$, that is
$$\Lambda_h:=\left\{ V : V \textnormal{ is } \F_{q^m}\textnormal{-subspace of } \F_{q^m}^k \textnormal{ and } \dim_{\Fm}(V)=h\right\}.$$

\begin{definition}
    Let $\mS$ be an $\Fmk$ system, and let $h,r$ be positive integers such that $1\le h \le k$. We say that $\mS$ is \textbf{$(h,r)$-evasive} if $\mS$ is $(\Lambda_h,r)$-\textbf{evasive}. When $h=r$, we say that $\mS$ is \textbf{$h$-scattered}. Finally, a $1$-scattered $q$-system will be simply called \textbf{scattered}. 
\end{definition}

The study of scattered $q$-systems originated in \cite{blokhuis2000scattered}. This notion was then generalized to any $h$ in \cite{csajbok2021generalising}, although the $(k-1)$-scatteredness was previously used in \cite{Sheekey2020} due to its connection with MRD codes. The complete correspondence between maximum $h$-scattered systems and MRD codes was then highlighted in \cite{zini2021scattered}. More recently, the more general property of evasiveness was analyzed in \cite{bartoli2021evasive} and \cite{gruica2022generalised}. The connection between evasiveness of $q$-systems and generalized rank weights of a corresponding code can be instead found in \cite{marino2023evasive}.

Finally, we introduce the following notion for a $q$-system. 
\begin{definition}\label{def:rhos}
    For an $\Fq$-subspace $V\subseteq \Fm^k$, we define the \textbf{$\Fm$-rank of $V$} to be the integer
    $$\rho(V)=\dim_{\Fm}(\langle V\rangle_{ \Fm}),$$
    that is the $\F_{q^m}$-dimension of the $\F_{q^m}$-span of any $\F_q$-basis of $V$.
    When we consider $\rho$ restricted to the $\Fq$-subspaces of an $\Fmk$ system $\mS$, we will write $\rho_\mS$.
\end{definition}

\begin{notation}\label{not:direct_sum}
Given a vector space $E$, we denote by $\mL(E)$ the lattice of all subspaces of $E$.
For any direct sum $E=E_1\oplus E_2\oplus\cdots\oplus E_t$ of $\F_q$-vector spaces $E_1, E_2, \ldots, E_t$, we denote by $\pi_i:E\longrightarrow E_i$  the corresponding projections and by $\iota_i:E_i\longrightarrow E$ the canonical embedding.
Furthermore, for $n_1,n_2,\ldots, n_t\in\mathbb{N}$ and $n=n_1+n_2+\cdots+n_t$, we set $\F_q^n=\F_q^{n_1}\oplus\F_q^{n_2}\oplus\cdots\oplus\F_q^{n_t}$ such that
$\pi_i:\F_q^n\longrightarrow \F_q^{n_i}$  is the projection onto the corresponding $n_i$ coordinates. Let $Q$ be a prime power and define $\iota_i:\mL(\mathbb{F}_Q^{n_i})\longrightarrow \mL(\mathbb{F}_Q^{n})$ such that for any $X \in \mL(\mathbb{F}_Q^{n_i})$ then $$\iota_i(X)=\langle 0 \rangle \oplus \cdots\oplus \langle 0\rangle \oplus X\oplus\langle 0 \rangle \oplus \cdots \oplus \langle 0 \rangle.$$
\end{notation}

\begin{definition}
    Let $\mS_1,\ldots,\mS_t$ be $[n_i,k_i]_{q^m/q}$ systems, for $1\leq i\leq t$. Then, denote by $\bigoplus_{i=1}^t\mS_i$ the $\left[\sum_{i=1}^t n_i, \sum_{i=1}^t k_i\right]_{q^m/q}$ system given by  $\iota_1(\mS_1)+\ldots+\iota_t(\mS_t)$. 
\end{definition}

\subsection{\emph{q}-Matroids}\label{subsec:preliminaries_q-mat}
In this subsection we briefly recall some facts about $q$-matroids that will be useful later on.
\begin{definition}
    A \textbf{$q$-matroid with ground space} $E$ is a pair $\mM=(E,\rho)$, where $E$ is a finite dimensional $\F_q$-vector space and $\rho$ is an integer-valued 
	function defined on $\mL(E)$ with the following properties:
	\begin{itemize}
		\item[(R1)] Boundedness: $0\leq \rho(A) \leq \dim A$, for all $A\in \mL(E)$. 
		\item[(R2)] Monotonicity: $A\leq B \Rightarrow \rho(A)\leq \rho(B)$, for all $A,B\in\mL(E)$. 
		\item[(R3)] Submodularity: $\rho(A+ B)+\rho(A\cap B)\leq \rho(A)+\rho(B)$, for all $A,B\in\mL(E)$.  
	\end{itemize}
 The function $\rho$ is called \textbf{rank function} and the value $\rho(\mM) := \rho(E)$ is the \textbf{rank of the $q$-matroid}. The value $h(\mM)=\dim_{\F_q}(E)$ is the \textbf{height of $\mM$}.
\end{definition}

A $1$-dimensional space $x\in\mL(E)$ is a \textbf{loop} if $\rho(x)=0$. A subspace $A\in\mL(E)$ is \textbf{independent} if $\rho(A)=\dim(A)$ and \textbf{dependent} otherwise. The inclusion-minimal dependent spaces are called \textbf{circuits}. A space $A\in\mL(A)$ is a \textbf{flat} if it is inclusion-maximal in the set $\{V \in \mL(E) \mid \rho(V)=\rho(A)\}$, and a space is \textbf{open} if it is the sum of circuits. Finally, a subspace $A\in\mL$ which is both a flat and an open space is called a \textbf{cyclic flat}. 

We denote by $\mI(\mM)$ and $\mZ(\mM)$ the collections of independent spaces and cyclic flats of the $q$-matroid $\mM$, respectively.
If it is clear from the context, we will simply write $\mI, \mZ $.

It is well known that the collection of independent spaces uniquely
determines the $q$-matroid, and the same is true for the collections of dependent spaces, open spaces,
and circuits. For this and many more cryptomorphisms for $q$-matroids we refer to \cite{byrne2022constructions}. Furthermore, the cyclic flats
along with their rank values  also uniquely determine the $q$-matroid; see \cite{alfarano2022cyclic}.

Two $q$-matroids $\mM_i=(E_i,\rho_i)$, $i=1,2$ are \textbf{equivalent} if there exists a rank-preserving $\Fq$-isomorphism $\alpha:E_1\to E_2$, i.e. $\rho_1(V)=\rho_2(\alpha(V))$ for all $V\in\mL(E_1)$. When this happens, we will write $\mM_1\cong\mM_2$.

We recall a well-known family of $q$-matroids, namely the family of \emph{uniform $q$-matroids}; see for instance \cite{jurrius2018defining}.

\begin{definition}\label{def:uniform}
Let $0\leq k\leq n$. For each $V\in\mL(\F_q^n)$, define $\rho(V) :=\min\{k,\dim(V)\}$. Then $(\F_q^n,\rho)$ is a $q$-matroid. It is called the \textbf{uniform $q$-matroid on $\F_q^n$ of rank $k$} and is denoted by $\mU_{k,n}(q)$.
\end{definition}

An important construction of a $q$-matroid arises from matrices; see \cite{jurrius2018defining, gorla2019rank}. Let $G$ be a $k \times n$ matrix with entries in $\F_{q^m}$ and for every $U\in\mL(\F_q^n)$, let $A^U$ be a matrix whose columns form a basis of $U$. Then the map 
\begin{equation}\label{eq:rank1}
    \rho:\mL(\F_q^n) \to \mathbb{Z}, \  U\mapsto \rk(G A^U),
\end{equation}
is the rank function of a $q$-matroid, which we denote by $\mM_G$ and we call the \textbf{$q$-matroid represented by $G$}. A $q$-matroid $\mM$ with ground space $\F_q^n$ is called \textbf{$\F_{q^m}$-representable} if $\mM=\mM_G$ for some matrix $G$ with  $n$ columns and entries in $\F_{q^m}$, whose rank equals the rank of $\mM$. The $q$-matroid $\mM$ is called \textbf{representable} if it is \textbf{$\F_{q^m}$-representable} over some field extension $\F_{q^m}$.
In other words, we can say that $\mM$ is $\F_{q^m}$-representable if there is a rank-metric code with generator matrix $G$, such that $\mM=\mM_G$. 

There is also an equivalent geometric interpretation of a representable $q$-matroid in terms of $q$-systems. This was proved in \cite[Theorem 4.6]{alfarano2022cyclic}, whose original statement may look slightly different.  {Indeed, there it was proved by using the independent spaces}. However, it is equivalent to the following reformulation,  {which uses the notion of rank function}.

 \begin{theorem}[{\cite[Theorem 4.6]{alfarano2022cyclic}}]\label{thm:independentrank}
Let $G$ be the generator matrix of a nondegenerate $\Fmk$ code, and let $\mS$ be any $\Fmk$ system associated to it.  {Then} $\mM_G\cong (\mS,\rho_{\mS})$.
\end{theorem}
 {\begin{proof}
     Let us consider the isomorphism of $\Fq$-vector spaces
    $$\psi_G: \Fq^n\longrightarrow \mS, \quad v \longmapsto v G^\top.$$
    For any given subspace $U\in \mathcal L(\Fq^n)$, let $A^U$ be a matrix whose columns form a basis of $U$. Then, $$\rho_\mS(\psi_G(U))=\dim_{\Fm}\langle \psi_G(U)\rangle_{\Fm}=\dim_{\Fm}(\mathrm{colsp}(GA^U))=\rk(GA^U)=\rho(U),$$
    where $\rho$ is the rank function defined in \eqref{eq:rank1} of the $q$-matroid $\mathcal M_G$.
\end{proof}}

In other words, Theorem \ref{thm:independentrank} characterizes representable $q$-matroids as those coming from a $q$-system. More precisely, a $q$-matroid $\mM$ of rank $k$ and height $n$ without loops is \textbf{$\Fm$-representable} if and only if it is equivalent to  a $q$-matroid $(\mS,\rho_\mS)$, for some $\Fmk$ system~$\mS$.  {Therefore, $(\mathcal{S},\rho_{\mathcal{S}})$ is called an $\mathbb{F}_{q^m}$-representation of $\mathcal{M}_G$.}

\begin{remark}[$\Fm$-independent spaces]\label{rem:independent_spaces}
    In the original statement \cite[Theorem 4.6]{alfarano2022cyclic}, the equivalence $\mM_G\cong (\mS,\rho_{\mS})$ is derived by characterizing the collection of independent spaces $\mI(\mM_G)$. This is  given by the set of \textbf{$\Fm$-independent}, which are precisely the $\Fq$-subspaces $I\subseteq \mS$ such that $\dim_{\Fq}(I)=\dim_{\Fm}(\langle I\rangle_{\Fm})$. Since by definition of $\rho_{\mS}$ (see Definition \ref{def:rhos}) the independent spaces of $(\mS,\rho_{\mS})$ coincides with $\mI(\mM_G)$, then Theorem \ref{thm:independentrank} is equivalent to \cite[Theorem 4.6]{alfarano2022cyclic}.
\end{remark}

\begin{remark}[Representability of uniform $q$-matroids] 
Let $\mU_{k,n}(q)$ be the uniform $q$-matroid of rank $k$ with ground space $\F_q^n$. Then the \emph{trivial} uniform $q$-matroid $\mU_{0,n}(q)$ and the \emph{free} uniform $q$-matroid $\mU_{n,n}$ are
representable over $\F_q$. In particular, $\mU_{0,n}(q)$ is represented by the $1 \times n$ zero matrix 
and $\mU_{n,n}(q)$ by
the $n\times n$ identity matrix. For $0 < k < n$, the uniform $q$-matroid $\mU_{k,n}(q)$ is representable over $\F_{q^m}$ if and only if
$m \geq n$. Indeed, a matrix $G\in\F_{q^m}^{k\times n}$ represents $\mU_{k,n}(q)$ if and only if $\rk(GY^\top) = k$ for all $Y \in \F_q^{k\times n}$ of rank $k$.
But this is equivalent to $G$ generating an MRD code and such a matrix $G$ exists
if and only if $m\geq n$; see e.g. \cite[Ex. 2.4]{gluesing2022representability}. 
Analogously, in the language of $q$-systems, a $\Fm$-representation of $\mathcal U_{k,n}(q)$ is given by any $\Fmk$ system which is $(k-1)$-scattered, and this is known to exist if and only if $m\ge n$; see e.g. \cite{Sheekey2020}.
\end{remark}

Like in the classical case, not all $q$-matroids are representable; see for instance the V\'amos $q$-matroid defined in \cite[Example 17]{byrne2021weighted} {, as shown in \cite{gluesing2021q}}. However, so far, very little is known about representability of $q$-matroids. 

\subsection{The direct sum of \emph{q}-matroids}
In \cite{ceria2021direct} the notion of direct sum of two $q$-matroids has been introduced. We are going to define the direct sum of $t$ many $q$-matroids iteratively, thanks to the associativity property of the direct sum, proved in \cite{gluesing2023decompositions}. We use Notation \ref{not:direct_sum}  {and the notation used in \cite{gluesing2023coproducts}}.

\begin{definition}
    Let $\mM_i=(E_i,\rho_i),\,i\in[t],$ be $q$-matroids and set $E=E_1\oplus\cdots\oplus E_t$.
Define $\rho'_i:\mL(E)\longrightarrow \mathbb{N}_0,\ V\longmapsto \rho_i(\pi_i(V))$ for $i\in[t]$.
Then $\mM'_i=(E,\rho'_i)$ is a $q$-matroid. 
We define the \textbf{direct sum of $\mM_1, \ldots,\mM_t$} to be the $q$-matroid $\mM:=(E,\rho)$, where $\rho$ is defined iteratively as follows:
\begin{itemize}
\item If $t=2$, then
\begin{equation}\label{e-rho}
\rho:\mL(E)\longrightarrow\mathbb{N}_0,\quad V\longmapsto \dim V+\min_{X\leq V}\big(\rho'_1(X)+\rho'_2(X)-\dim X\big).
\end{equation}
\item If $t>2$, then $\rho$ is the rank function of $(\mM_1\oplus\cdots\oplus\mM_{t-1})\oplus \mM_t$.
\end{itemize}
\end{definition}

In \cite{gluesing2022representability}, the authors initiated the study of the representability of the direct sum of two $q$-matroids. In particular, they proved that if the direct sum is representable then a representation is given by a block diagonal matrix whose blocks represent the
summands. In the following, we list a few important results from \cite{gluesing2022representability}, reformulated in the language of $q$-systems.

\begin{theorem}[{\cite[Theorem 3.4]{gluesing2022representability}}]\label{thm:GL-J-representation}
Let $\mM_i=(\F_q^{n_i},\rho_i),i=1,2,$ be $q$-matroids of rank $k_i$, and let  $\mM=\mM_1\oplus\mM_2$.
Suppose $\mM$ is representable over $\F_{q^m}$.
Then $\mM_1$ and $\mM_2$ are representable over $\F_{q^m}$ and $\mM=(\mS_1\oplus \mS_2,\rho_{\mS_1\oplus\mS_2})$ for some $[n_i,k_i]_{q^m/q}$ systems $\mS_i$, $i \in \{1,2\}$.
Furthermore, $(\mS_i,\rho_{\mS_i})$ is an $\Fm$-representation of $\mM_i$ for each $i \in \{1,2\}$. 
\end{theorem}

Moreover, in \cite{gluesing2022representability}, it is shown that the direct sum of two $\F_{q^m}$-representable $q$-matroids may not be $\F_{q^m}$-representable or, even, not be representable over any field extension of $\F_q$. 

\begin{example}[{\cite[Proposition 3.8]{gluesing2022representability}}]\label{exa:representable_U12}
Let $\mM = \mU_{1,2}(q)$. Then $\mM$ is representable over $\F_{q^2}$, whereas $\mM\oplus \mM$ is representable over $\F_{q^m}$ if and only if $m \geq 4$.
\end{example}

\begin{example}[{\cite[Proposition 3.7]{gluesing2022representability}}]\label{ex:nonrepr}
    Let $\F_q=\F_2$ and $\F_4=\{0,1,\omega,\omega+1\}$. Consider the full $[4,2]_{4/2}$ system 
$ \mS=\F_4^2$
and set $\mM:=(\mS,\rho_{\mS})$.
Then $\mM\oplus\mM$ is not representable.
\end{example}

In the rest of the paper we investigate the representability of the direct sum of $t$ uniform $q$-matroids. To this end, we will use the following properties of cyclic flats of $q$-matroids.

\begin{lemma}[{\cite[Lemma 2.28]{alfarano2022cyclic}}]\label{lem:indep_cyclic_flats}
Let $\mM=(E,\rho)$ be a $q$-matroid. $I\in \mL(E)$ is independent if and only if for every cyclic flat $X\in\mL(E)$, $\dim(I \cap X) \leq \rho(X)$.
\end{lemma}

\begin{lemma}[{\cite[Proposition 2.30]{alfarano2022cyclic}}]\label{lem:uniform_cyclic_flats}
Let $\mM=(\F_q^n,\rho)$ be a $q$-matroid of rank $k$, with $0 < k < n$. Then $\mM =\mU_{k,n}(q)$ if and only if  {$\langle 0\rangle$ and $\F_q^n$ are the only cyclic flats of $\mM$}.
\end{lemma}

In \cite{gluesing2023decompositions}, it ha been shown that cyclic flats also behave well with the direct sum of $q$-matroids. We state the result for $t$ summands.

\begin{lemma}[{\cite[Theorem 6.2]{gluesing2023decompositions}}]\label{lem:sum_cyclic_flats}
    For each $i\in [t]$, let $\mM_i=(\F_q^{n_i},\rho_i)$ be a $q$-matroid. Then,     $$\mZ(\mM_1\oplus\cdots\oplus\mM_t)=\left\{\bigoplus_{i=1}^tZ_i\,:\, Z_i\in\mZ(\mM_i), \; \textnormal{ for } i\in[t]\right\}.$$
\end{lemma}


\section{The direct sum of uniform \emph{q}-matroids}\label{sec:direct_sum_uniform}

In this section we take the first steps toward the geometric characterization of the representability of the direct sum of $t$ uniform $q$-matroids  {$\mU_{k_1,n_1}(q)\oplus\cdots\oplus\mU_{k_t,n_t}(q)$}, by characterising its independent spaces. In the rest of the paper we will assume that none of the $t$ summands is the trivial or the free uniform $q$-matroid. This is due to the following observation.

\begin{remark}
The trivial uniform $q$-matroid $\mU_{0,s}(q)$ has only loops, so we cannot consider its representation as $q$-system. However, as we already pointed out, in terms of matrices,  $\mU_{0,s}(q)$  can represented by a $1\times s$ zero matrix. Moreover, it is not difficult to see that a matrix $G\in \Fm^{k\times n}$ is an $\Fm$-representation of a $q$-matroid $\mM$ if and only if the matrix 
$$(\,G\,|\,0\,)\in \Fm^{k \times(n+s)}$$
is an $\Fm$-representation of $\mM\oplus\mU_{0,s}(q)$. 

 Furthermore, in \cite[Proposition 3.9]{gluesing2022representability} it is proven that a $q$-matroid $\mM$ is $\F_{q^m}$-representable as $\mM=(\mS,\rho_{\mS})$, if and only if $\mM\oplus \mU_{s,s}(q)$ is $\F_{q^m}$-represented by $(\mS\oplus \Fq^s,\rho_{\mS\oplus \Fq^s})$. 

Hence, without loss of generality, we can assume from now on that none of the summands involved is the trivial uniform $q$-matroid nor the free uniform $q$-matroid. 
\end{remark}

For the rest of the paper, assume  {$1\leq k_i<n_i$} and let $\mU_{k_i,n_i}(q)$, $i\in[t]$ be uniform $q$-matroids. 

By Lemmas \ref{lem:uniform_cyclic_flats} and \ref{lem:sum_cyclic_flats}, we have that
$$\mZ(\mU_{k_1,n_1}(q)\oplus\cdots\oplus \mU_{k_t,n_t}(q))=\left\{\bigoplus_{i=1}^t Z_i \,:\, Z_i\in\{\langle 0 \rangle, \F_q^{n_i}\}, i \in [t]\right\}.$$

Moreover, it follows from \cite[Theorem 45]{ceria2021direct} that
$$\rho\left(\bigoplus_{i=1}^t Z_i\right)=\sum_{i=1}^t\rho_i(Z_i),$$
where $Z_i\in\{\langle 0 \rangle, \F_q^{n_i}\}$. 

The next theorem finds the independent spaces of the direct sum of $t$ uniform $q$-matroids in terms of the cyclic flats of the direct sum.

\begin{theorem}\label{thm:independent_direct_sum}
     {Let $k_1,\ldots, k_t,n_1,\ldots, n_t$ be positive integers, with $1\le k_i< n_i$ for $i \in [t]$.}
    Let $\mM=\mU_{k_1,n_1}(q)\oplus \cdots\oplus\mU_{k_t,n_t}(q)$ and for each $\mathcal{J}\subseteq [t]$ denote $k_{\mathcal J}:=\sum_{j\in \mathcal J}k_j$. Then $I\in\mI(\mM)$ if and only if
    $$ \dim\Big(I \cap \Big(\sum_{j \in \mathcal J} \iota_j\left(\Fq^{n_j}\right)\Big)\Big)\le k_{\mathcal J},  \qquad \mbox{ for every } \mathcal J \subseteq [t].$$
\end{theorem}
\begin{proof}
    It follows directly from Lemmas \ref{lem:indep_cyclic_flats} and \ref{lem:sum_cyclic_flats} and the discussion above.
\end{proof}

Let $\mathbf{k}=(k_1,\ldots,k_t)\in\mathbb N^t$ and let $h,k$ be  positive integers such that $k=k_1+\ldots+k_t$ and $1\leq h \leq k-1$, then 
we define $$\Lambda_{h,\mathbf{k}}:=\left\{ V : V \textnormal{ is } \F_{q^m}\textnormal{-subspace of } \F_{q^m}^k,  \dim_{\Fm}(V)=h \textnormal{ and } \iota_i(\Fm^{{k_i}})\not\subseteq V \mbox{ for } i\in [t]\right\}.$$ We need the following technical lemma for the proof of the main result of this section.

\begin{lemma}\label{lem:boxpropgen}
Let $k_1,\ldots,k_t,k,n_1,\ldots,n_t,m,r$ be positive integers, with $1\le k_i< n_i\le m$ for $i \in [t]$ and $k=k_1+\ldots+k_t$. Let $\mathcal{S}_i$ be a $[n_i,k_i]_{q^m/q}$ system for $i \in [t]$.
If  $\mathcal{S}=\mathcal{S}_1\oplus\ldots\oplus\mathcal{S}_t$ is $(\Lambda_{k-1,\mathbf{k}},r)$-evasive, then $\mathcal{S}$ is $(\Lambda_{l,\mathbf{k}},r-k+1+l)$-evasive for any $l \in [k-1]$.
 {
In particular, if $\mS$ is $(\Lambda_{k-1,\mathbf{k}},k-1)$-evasive, then $\mathcal{S}$ is $(\Lambda_{l,\mathbf{k}},l)$-evasive for any $l \in [k-1]$.
}
\end{lemma}
\begin{proof}
We prove the statement by induction on $k-1-l$. 

If $k-1-l=0$, then there is nothing to prove and the assertion is clearly true. 

Suppose that $l<k-1$ and assume that the claim holds true for $l+1$.
Let us take any subspace $T \in \Lambda_{l,\mathbf{k}}$.  We will show that there exists a subspace $T'\in \Lambda_{l+1,\mathbf{k}}$ containing $T$ and such that
\begin{equation}\label{eq:condbox2} 
\dim_{\Fq}(T'\cap \mathcal{S})\geq \dim_{\Fq}(T\cap \mathcal{S})+1. 
\end{equation}
This will imply the statement, since 
$$\dim_{\Fq}(T\cap \mathcal{S})\leq \dim_{\Fq}(T'\cap \mathcal{S})-1 \le r-k+1+l+1-1=r-k+1+l.$$

We divide the proof in two cases.

\noindent \underline{\textbf{Case I.}} First, suppose that there exists $i \in [t]$ such that 
\[ \dim_{\F_{q^m}}(T\cap \iota_i( \Fm^{{k_i}}))<k_i-1.\]
Since $\mathcal{S}_i$ is an $[n_i,k_i]_{q^m/q}$ system contained in $\iota_i( \Fm^{{k_i}})$ there exists $v \in \mathcal{S}_i\setminus T$. Clearly, $T'=T+ \langle v \rangle_{\Fm} \in \Lambda_{l+1,\mathbf{k}}$ and satisfies \eqref{eq:condbox2}. 

\noindent \underline{\textbf{Case II.}} Now assume that for  {every} $i \in [t]$ we have that 
\[ \dim_{\F_{q^m}}(T\cap \iota_i( \Fm^{{k_i}}))=k_i-1.\]
So, $l\geq k-t$ and $T\supseteq (T\cap \iota_1( \Fm^{{k_1}}))+ (T\cap \iota_2( \Fm^{{k_2}}))+\ldots+(T\cap \iota_t( \Fm^{{k_t}}))$. 

{For every $i \in \{2,\ldots,t\}$, consider $T_i=\iota_1( \Fm^{{k_1}})+\ldots+\iota_i( \Fm^{{k_i}})$. Then, $\dim_{\Fm}(T_i)=k_1+\ldots +k_i$ and, since  {$T\cap T_i\supseteq (T\cap T_{i-1})+(T\cap \iota_i(\Fm^{k_i}))$, }
\begin{equation}\label{eq:dimTcapTi}
    \dim_{\Fm}(T\cap T_i)\in\{ \dim_{\Fm}(T\cap T_{i-1})+k_i-1, \ldots, k_1+\ldots+k_{i}-1\}. 
\end{equation} 
We further distinguish two subcases.

\noindent \underline{\textbf{Case II.a.}}
Assume that there exists $i \in \{1,\ldots, t\}$ such that $\dim_{\Fm}(T\cap T_{i})=\dim_{\Fm}(T\cap T_{i-1})+k_{i}-1$. Hence, 
 {\[ T\cap T_i=(T\cap T_{i-1})+(T \cap \iota_i( \Fm^{{k_i}})). \]
Let $v_1 \in \mathcal{S}_1\oplus\ldots\oplus \mS_{i-1}$ and $v_i\in \mathcal{S}_i$ such that $v_1,v_i \notin T$. Note that such vectors exist, otherwise $T$ would contain at least one of the $\mathcal{S}_1,\ldots,\mathcal{S}_i$, a contradiction as $T \in \Lambda_{l,\mathbf{k}}$.
Also, we have that $\langle v_1,v_i\rangle_{\Fm} \cap T=\{0\}$. Indeed, if $\alpha v_1 +\beta v_i \in \langle v_1,v_i\rangle_{\Fm} \cap T$ with $\alpha,\beta\in\F_{q^m}$, and $(\alpha,\beta)\ne(0,0)$, then $\alpha v_1 \in T\cap T_{i-1}$ and $\beta v_i \in T \cap \iota_i( \Fm^{{k_i}})$, a contradiction as $v_1,v_i \notin T$.}
Therefore, $T'=T\oplus \la v_1+v_i\ra_{\F_{q^m}} \in \Lambda_{l+1,\mathbf{k}}$ and satisfies \eqref{eq:condbox2}.

\noindent \underline{\textbf{Case II.b.}} Suppose now that for every $i \in \{2,\ldots,t\}$ we have  {$\dim_{\Fm}(T\cap T_i)\geq \dim_{\Fm}(T\cap T_{i-1})+k_i$.} Observe that this implies that $\dim_{\Fm}(T\cap T_{2})\geq \dim_{\Fm}(T\cap T_1)+k_2=k_1+k_2-1$, and inductively, for every $i\in [t]$,  $\dim_{\Fm}(T\cap T_i)\geq k_1+\ldots+k_{i}-1$, meaning equality due to \eqref{eq:dimTcapTi}. Hence, the equality also holds for $i=t$, that is, 
$$\dim_{\Fm}(T\cap T_t)=k_1+\ldots+k_t-1=k-1,$$
which is a contradiction to the hypothesis that $\dim_{\Fm}(T)=l<k-1$.
}
\end{proof}

The next result is the main theorem of this section. It provides a geometric characterization of the independent spaces of the direct sum of $t$ many uniform $q$-matroids.


\begin{theorem}\label{thm:charact_independence}
Let $k_1,\ldots,k_t,k,n_1,\ldots,n_t, n, m$ be positive integers, with $1\le k_i< n_i\le m$ for $i \in [t]$ and $k=k_1+\ldots+k_t$, $n=n_1+\ldots+n_t$. Let $\mS_i$ be an $\Fm$-representation for $\mU_{k_i,n_i}(q)$ for each $i\in [t]$. Define the $\Fmk$ system
$$\mS:=\bigoplus_{i=1}^t\mS_i.$$
Moreover, for each $\mathcal J\subseteq [t]$, denote by $k_{\mathcal J}:=\sum_{j\in \mathcal J}k_j$ and $\mS_{\mathcal J}:=\sum_{j\in \mathcal J}\iota_j(\mS_j)$.
Then the following statements are equivalent.
\begin{enumerate}
    \item $(\mS,\rho_{\mS})$ is an $\Fm$-representation for $\mU_{k_1,n_1}(q)\oplus\ldots\oplus\mU_{k_t,n_t}(q)$.
    \item For every $\F_q$-subspace $I\subseteq \mS$ it holds that
    \begin{equation}\label{eq:main_th_indep_general}
        \rho_{\mS}(I)=\dim_{\F_q} (I) \Longleftrightarrow 
            \dim_{\Fq}(I\cap \mS_{\mathcal J})\le k_{\mathcal J}, \; \forall \; \mathcal J\subseteq[t].
    \end{equation}
    \item $\mathcal{S}$ is $(\Lambda_{k-1,\mathbf{k}},k-1)$-evasive, where $\mathbf{k}=(k_1,\ldots,k_t)$.
\end{enumerate}
\end{theorem}

\begin{proof} \, 
  \noindent
        ~{\bf} \underline{$(1) \Longleftrightarrow (2)$}: Let us call $\mathcal M_1=(\mS,\rho_{\mS})$ and $\mathcal M_2=\mU_{k_1,n_1}(q)\oplus\ldots\oplus \mU_{k_t,n_t}(q)$. Then, $\mM_1\cong \mM_2$ if and only if there exists an invertible $\Fq$-linear map $\psi:\Fq^{n_1+\ldots+n_t}\longrightarrow\mS$ such that $\mI(\mM_1)=\mI(\psi(\mM_2))$. Let $G$ be any generator matrix associated to the $\Fmk$ system $\mS$ of the form
        $$G=\begin{pmatrix}
            G_1 & 0 & & 0\\
            0 & G_2 &\phantom{\ddots}
            \\
             &  & \ddots & 0 \\
            0 & \phantom{\ddots} & 0 & G_t 
        \end{pmatrix},$$
        where $G_i\in \Fm^{k_i\times n_i}$, is a generator matrix associated with $\mS_i$, for each $i \in [t]$. Then, using the characterization of independent spaces for $\mM_1$ given in Remark \ref{rem:independent_spaces}, and the characterization of independent spaces for $\mM_2$ derived in Theorem \ref{thm:independent_direct_sum}, together with the map 
        $\psi=\psi_G$ as in Theorem \ref{thm:psiG}, we conclude. 
        
\noindent
    ~{\bf} \underline{$(2) \Longrightarrow (3)$}: Assume towards a contradiction that $\mathcal{S}$ is not $(\Lambda_{k-1,\mathbf{k}},k-1)$-evasive. Then, there exists an $\F_{q^m}$-hyperplane $\bar{H}\subseteq\F_{q^m}^{k}$ such that $\langle \mS_i\rangle_{\F_{q^m}}\not\subseteq \bar{H}$ for each $i\in [t]$, and $\dim_{\F_q} (\bar{H}\cap \mS) \geq k$. Let $T:=\bar{H}\cap \mS$ and observe that $\rho_{\mS}(T)<k\le \dim_{\Fq}(T)$. If $T$ is such that $\dim_{\Fq}(T\cap \mS_{\mathcal J})\le k_{\mathcal J}$ for every $\mathcal J \subseteq [t]$,  then $T$ contradicts Eq. \eqref{eq:main_th_indep_general}. Otherwise, let $\overline{\mathcal J}$ be a minimal subset of $[t]$ such that $\dim_{\Fq}(T\cap \mS_{\overline{\mathcal J}})>k_{\overline{\mathcal I}}$, and take any $I\subseteq T\cap \mS_{\overline{\mathcal J}}$ with $\dim_{\Fq}(I)=k_{\overline{\mathcal J}}$. For every $j \in [t]\setminus \overline{\mathcal J}$, select an independent space $I_j\subseteq \mS_{j}$ with $\dim_{\Fq}(I_j)=\rho_{\mS_j}(I)=\langle I_j\rangle_{\Fm}=k_j$. It can be easily verified that, by construction, the space 
    $$ \widetilde{I}:=I\oplus \bigoplus_{j \notin \overline{\mathcal J}}\iota_j(I_j)$$
    is such that 
    $\dim_{\Fq}(\widetilde{I})=k>\rho_{\mS}(\widetilde{I})$,
    contradicting Eq. \eqref{eq:main_th_indep_general}.

    \noindent
        ~{\bf} \underline{$(3) \Longrightarrow (2)$}: 
    First of all, we observe that the implication ``$\Longrightarrow$'' in Eq. \eqref{eq:main_th_indep_general} is always true. Indeed, if $\dim_{\Fq}(I\cap \mS_{\mathcal J})>k_{\mathcal J}$ for some $\mathcal J\subseteq [t]$, then, since $\rho_{\mS_{\mathcal J}}(I\cap \mS_{\mathcal J})\le k_{\mathcal J}$, we have that $I\cap \mS_{\mathcal J}$ is not $\Fm$-independent, and hence $I$ cannot be $\Fm$-independent. Therefore, $\rho_{\mS}(I)<\dim_{\Fq}(I)$, yielding a contradiction.
    
    Thus, we only need to show the opposite implication. 
    Let $I$ be an $\Fq$-subspace of $\mS$ and let $\Gamma:=\langle I\rangle_{\Fm}$.   
          Observe that if $\Gamma \supseteq \iota_i( \Fm^{{k_i}})$ for each $i \in [t]$, then $\Gamma=\Fm^k$, and $\rho_{\mS}(I)=k$. Hence, since we have by hypothesis that $\dim_{\Fq}(I)=\dim_{\Fq}(I\cap \mS_{[t]})\le k_{[t]}=k$, we automatically obtain equality, and $\rho_{\mS}(I)=k=\dim_{\Fq}(I)$. Therefore, we may assume that $\Gamma$ does not contain all the spaces $\iota_i( \Fm^{{k_i}})$.

        Assume by contradiction that \eqref{eq:main_th_indep_general} is not satisfied. Then there exists an $\Fq$-subspace $I\subseteq \mS$ such that $\dim_{\Fq}(I)\le k$, $\dim_{\Fq}(I\cap \mS_{\mathcal J}) \leq k_\mathcal J$ for each $\mathcal J \subseteq [t]$ and $\rho_{\mS}(I)<\dim_{\Fq}(I)$. 
        We divide the proof in two cases:
        
        \noindent \underline{\textbf{Case I.}} If $\Gamma$ does not contain any of the $\Fm$-subspaces $\iota_i( \Fm^{{k_i}})$, then $\ell:=\rho_{\mS}(I) <k$ and $\Gamma\in \Lambda_{\ell,\mathbf{k}}$. This  {implies that $\mS$ is not $(\Lambda_{\ell,\mathbf{k}},\ell)$-evasive, contradicting the hypothesis of $\mS$ being $(\Lambda_{k-1,\mathbf{k}},k-1)$-evasive}, due to Lemma \ref{lem:boxpropgen}.        Furthermore, this also shows that for every $\Fq$-subspace $I'\subseteq\mS$, if $\Gamma':=\langle I'\rangle_{\Fm}$ does not contain any space of the form $\iota_i( \Fm^{{k_i}})$ for $i \in [t]$, then \begin{equation}\label{eq:def_rho_S}         \rho_{\mS}(I')=\min\{k,\dim_{\Fq}(I')\},     \end{equation}      {otherwise, as before, we can apply again Lemma \ref{lem:boxpropgen} and get a contradiction.}
        {In fact, we just proved that if $\rho_\mS(I')<\dim_{\Fq}(I')$ then $\rho_\mS(I')=k$.}
        
        \noindent \underline{\textbf{Case II.}} 
        Up to a permutation of the set $[t]$, we can assume without loss of generality  that $\Gamma$ contains the $\Fm$-subspaces $\iota_i( \Fm^{{k_i}})$, for each $i\in[a]$, for some  $a$ such that $0<a<t$ (note that $a=0$ is Case I, and $a=t$ was already proved before). Let  {$\dim_{\Fq}(I)=r+s+d$} with $r=\dim_{\Fq}(I\cap \mS_{[a]} )\leq k_{[a]}$ and $s=\dim_{\Fq}(I\cap \mS_{[t]\setminus[a]})\leq k_{[t]\setminus[a]}$. $I$ can be written as $$I=\langle (v_1 \mid 0),\ldots, (v_{r} \mid 0)\rangle_{\Fq}+\langle (0  \mid u_1),\ldots, (0 \mid u_s)\rangle_{\Fq}+\langle (w_1 \mid z_1),\ldots, (w_d \mid z_d)\rangle_{\Fq}.$$ 
        Moreover, the following hold:
        \begin{enumerate}
            \item $X_1:=\langle v_1,\ldots,v_{r},w_1,\ldots,w_d\rangle_{\Fq}  \subseteq \mS_{[a]}$ and $\dim_{\Fq}(X_1)=r+d$; 
            \item $X_2:=\langle u_1,\ldots,u_{s},z_1,\ldots,z_d\rangle_{\Fq} \subseteq \mS_{[t]\setminus[a]}$ and $\dim_{\Fq}(X_2)=s+d$. 
        \end{enumerate}
        Note that $\Gamma=\langle I\rangle_{\Fm}=\bigoplus_{i \in [a]}\iota_i(\F_{q^m}^{k_i})+(\langle0\rangle \oplus \langle X_2\rangle_{\Fm})$,
     where $\langle X_2\rangle_{\Fm}$ does not contain any of the spaces of the form $\iota_i(\Fm^{k_i})$ for $i \in [t]\setminus [a]$. Thus, by \eqref{eq:def_rho_S} it must hold that \begin{equation}\label{eq:inductive}\rho_{\mS_{[t]\setminus[a]}}(X_2)=\min \{ k_{[t]\setminus[a]},\dim_{\F_q}(X_2) \}.\end{equation}
        Moreover, by the form of $\Gamma$, we have        $$\rho_{\mS}(I)=k_{[a]}+\dim_{\Fm}(\langle X_2\rangle_{\Fm})=k_{[a]}+\rho_{\mS_{[t]\setminus[a]}}(X_2).$$
        Combining this with the assumption that $r+s+d=\dim_{\Fq}(I)>\rho_{\mS}(I)$, we obtain
        $$ \rho_{\mS_{[t]\setminus[a]}}(X_2)<s+d+r-k_{[a]}\leq k-k_{[a]}= k_{[t]\setminus[a]}.$$
        In addition, $ \rho_{\mS_{[t]\setminus[a]}}(X_2)<s+d+r-k_{[a]}\le  s+d =\dim_{\F_q}(X_2)$.
        These two last observations contradict \eqref{eq:inductive}, since they imply $\rho_{\mathcal{S}_{[t]\setminus[a]}}(X_2)< \min \{ k_{[t]\setminus[a]},\dim_{\F_q}(X_2) \}$.
\end{proof}

\begin{example}
Let $\mM = \mU_{1,2}(q)$. Then $\mM$ is representable over $\F_{q^2}$, while it was shown in \cite[Theorem 3.4]{gluesing2022representability} that $\mM\oplus \mM$ is representable over $\F_{q^m}$ if and only if $m \geq 4$. Let $q=2$, $m\ge4$, and let $\F_{2^4}=\F_2(\alpha)$. We now show using Theorem \ref{thm:charact_independence} that a representation of $\mM\oplus \mM$ is given by the $[4,2]_{2^4/2}$ system 
$$\mS = \{(a,b)\in\F_{2^m}^2 \; : \; a\in\langle 1,\alpha \rangle_{\F_2}, b\in\langle 1,\alpha^2 \rangle_{\F_2}\}\leq \F_{2^m}^4.$$
Let $(a_1,b_1), (a_2,b_2)\in \mS$, with $a_1,a_2,b_1,b_2\ne 0$, and assume that $(a_1,b_1)=\gamma (a_2,b_2)$ for some $\gamma \in \F_{2^m}$. Then $a_1b_2=a_2b_1$. By definition of $\mS$, for each $i\in\{1,2\}$ we can write
$$a_i=\lambda_{i,1}+\lambda_{i,2}\alpha, \quad b_i=\mu_{i,1}+\mu_{i,2}\alpha^2,$$
for some $\lambda_{i,j},\mu_{i,j}\in \F_2$, which implies
\begin{equation}\label{eq:example}(\lambda_{1,1}+\lambda_{1,2}\alpha)(\mu_{2,1}+\mu_{2,2}\alpha^2)= (\lambda_{2,1}+\lambda_{2,2}\alpha)(\mu_{1,1}+\mu_{1,2}\alpha^2).\end{equation}
Since $m\geq 4$ we have that $1,\alpha,\alpha^2,\alpha^3$ are $\F_2$-linearly independent, and by equating their coefficients in \eqref{eq:example}, we obtain
$$\begin{cases}
    \lambda_{1,1}\mu_{2,1}=\lambda_{2,1}\mu_{1,1}, \\
    \lambda_{1,2}\mu_{2,1}=\lambda_{2,2}\mu_{1,1}, \\
    \lambda_{1,1}\mu_{2,2}=\lambda_{2,1}\mu_{1,2}, \\
    \lambda_{1,2}\mu_{2,2}=\lambda_{2,2}\mu_{1,2}. 
\end{cases}$$
Using the further assumption that $(\lambda_{i,1},\lambda_{i,2})\neq (0,0)$ and $(\mu_{i,1},\mu_{i,2})\neq (0,0)$ for $i \in \{1,2\}$  the only possible solutions satisfy  
$$\begin{cases}\lambda_{1,j}=\lambda_{2,j} & \mbox{ for } j \in \{1,2\},\\
\mu_{1,j}=\mu_{2,j} & \mbox{ for } j \in \{1,2\}.
\end{cases}$$
In other words, we get $\gamma=1$ and $(a_1,b_1)=(a_2,b_2)$. This shows that, for every $(a,b) \in \mS$ with $a,b \neq 0$, we have
$$\dim_{\F_{2}}(\mathcal{S}\cap\langle (a,b)\rangle_{\F_{2^4}})=1,$$
and hence $\mS$ is $(\Lambda_{(1,1),1},1)$-evasive. Thus, $\mS$ is an $\F_{2^4}$-representation of $\mM\oplus \mM$. We will see a much more general version of this constructive proof for the representability of $\mU_{1,n_1}(q)\oplus \mU_{1,n_2}(q)$ in Proposition \ref{prop:m_ge_n1n2}.
\end{example}

\section{A representation of the direct sum of uniform \emph{q}-matroids}\label{sec:construction}

In this section we will show that the direct sum of $t$ many uniform $q$-matroids is always representable, independent of their ranks and their heights. We will do this by using  the following arguments.
First of all, we notice that, by Theorem \ref{thm:GL-J-representation}, if such a  direct sum is representable, then it must be represented by the direct sum of the representations of its summands. Then, we  construct $q$-systems $\mS_1,\ldots,\mS_t$ whose direct sum satisfies Theorem \ref{thm:charact_independence}(3).

We remark that the direct sums of $q$-systems and their connection with rank-metric codes have been studied in \cite{bartoli2022exceptional}, with the name of \textbf{decomposable} $q$-systems, and in \cite{adriaensen2023minimum}, with the name of $q$-systems with \textbf{complementary subspaces}.

\begin{theorem}\label{thm:direct_sum_constr_general} 
Let $k_1, \ldots k_t, n_1,\ldots  n_t$ be positive integers such that $\gcd(n_i,n_j)=1$ for each $i\neq j$,  $k_i< n_i$ for each $i \in [t]$. Let $m=n_1\cdot\ldots\cdot n_t$ and let
\[\mS_i:=\{ (x,x^q,\ldots,x^{q^{k_i-1}}) \colon x \in \F_{q^{n_i}}\}, \qquad \mbox{ for each } i \in [t]\]
Then, $\mS:=\mS_1\oplus\ldots\oplus\mS_t$ is a $(\Lambda_{k-1,\mathbf{k}},k-1)$-evasive $\Fmk$ system, with $n=n_1+\ldots+n_t$, $k=k_1+\ldots+k_t$ and $\mathbf{k}=(k_1,\ldots,k_t)$.
\end{theorem}
\begin{proof}
 We first introduce the following notation. For a given vector $v=(v_1,\ldots,v_\ell)\in \Fm^\ell$, define the linearized polynomial
$$v(x):=\sum_{j=1}^\ell v_jx^{q^{j-1}},$$
and the vector $v^q:=(v_1^q,\ldots,v_\ell^q)$.
Since $k_i< n_i$, then $\mS_i$ is an $[n_i,k_i]_{q^m/q}$ system, for each $i\in [t]$. Thus, $\mS$ is an $\Fmk$ system.
Let us prove that $\wt_\mS(H)\le k-1$ for each hyperplane $H$ not containing any of the spaces $\iota_i( \Fm^{{k_i}})$, for $i \in [t]$. Observe that such hyperplane will be of the form $H=(a_{1,1},\ldots,a_{1,k_1},a_{2,1},\ldots,a_{2,k_2},\ldots,a_{t,1},\ldots,a_{t,k_t})^\perp$, for some nonzero vectors $a_i=(a_{i,1},\ldots,a_{i,k_i})\in \Fm^{k_i}$, for $i\in [t]$. Equivalently, we will prove that 
\begin{equation}\label{eq:card_HcapS_gen}
    |T|\le q^{k-1},
\end{equation}
where 
$$T=\left\{(x_1,\ldots,x_t)\in \F_{q^{n_1}}\times\ldots\times  \F_{q^{n_t}} \colon \sum_{r=1}^ta_r(x_r)=0\right\}.$$

We will show this using double induction  {on $t$ and $k_1$}. We first start our induction on $t$.

If $t=1$, then $|T|\le q^{k_1-1}$, since $\mS_1$ is a $(k_1-1)$-scattered $[n_1,k_1]_{q^m/q}$ system.

Now assume that the claim holds for every $t'<t$.
We proceed by induction on $k_1$. 
For $k_1=1$, up to dividing by $a_{1,1}$, we can assume that $H=(1\mid a_2 \mid \ldots \mid a_t)^\perp$.
Thus,
$$T=\left\{(x_1,\ldots,x_t) \in \F_{q^{n_1}}\times \ldots\times \F_{q^{n_t}} \colon x_1+\sum_{r=2}^ta_{r}(x_r)=0 \right\}.$$
Since $\gcd(n_i,n_t)=1$ for every $i \in [t-1]$, there exists $h\in \mathbb N$ such that $h\equiv 1 \pmod{n_t}$ and $h\equiv 0\pmod{n_i}$ for every $i \in [t-1]$. Therefore, if $(x_1,\ldots,x_t)\in T$, then
$$\sum_{r=2}^{t}a_{r}(x_r)=-x_1=-x_1^{q^{h}}=\sum_{r=2}^{t-1}a_{r}^{q^{h}}(x_r)+a_{t}^{q^{h}}(x_t^q).$$
Thus, 
\begin{equation}\label{eq:system1_gen}\begin{cases}
\displaystyle\sum\limits_{r=2}^{t}d_{r}(x_r)=0, \\
x_1+\displaystyle\sum\limits_{r=2}^ta_{r}(x_r)=0,
\end{cases}\end{equation}
where $d_{r}=a_{r}-a_{r}^{q^{h}}$ for each $r \in \{2,\ldots,t-1\}$,
$d_{t,i}=a_{t,i}-a_{{t,i-1}}^{q^{h}}$ for each $i \in [k_t+1]$ and let $a_{t,k_t+1}=a_{t,0}=0$. It is readily seen that $d_t=(d_{t,1},\ldots,d_{t,k_t+1})\neq 0$.
Let $\mathcal J:=\{i \in \{2,\ldots,t-1\} \,:\, d_i=0\}$. We have two cases.

\noindent \underline{\textbf{Case I:}} $\mathcal J=\emptyset$. Then, all the vectors $d_i$ are nonzero for each $i \in \{2,\ldots,t\}$, and 
 {$$T\subseteq T':=\left\{\left(-\sum_{r=2}^ta_r(x_r),x_2,\ldots,x_t\right)\in \F_{q^{n_1}}\times\ldots\times  \F_{q^{n_t}} \colon \sum_{r=2}^td_r(x_r)=0\right\}.$$}
{Since $T'$ has the same cardinality of the set
$$ \left\{(x_2,\ldots,x_t)\in \F_{q^{n_2}}\times\ldots\times  \F_{q^{n_t}} \colon \sum_{r=2}^td_r(x_r)=0\right\},$$ which only has $t-1$ components, we can use  induction hypothesis on $t'=t-1$, and derive that $|T'|\le q^{k_2+\ldots+(k_t+1)-1}=q^{k_1+\ldots+k_t-1}$.}

\noindent \underline{\textbf{Case II:}} $\mathcal J\neq \emptyset$. Since $d_t \neq 0$, up to reordering the variables, we may assume $\mathcal J=\{2,\ldots,s\}$, for some $s<t$. Then, \eqref{eq:system1_gen} reads as 
\begin{equation}\label{eq:system1_gen_case2}\begin{cases}
\displaystyle\sum\limits_{r=s+1}^{t}d_{r}(x_r)=0, \\
x_1+\displaystyle\sum\limits_{r=2}^ta_{r}(x_r)=0,
\end{cases}\end{equation}
By inductive hypothesis, the first equation in \eqref{eq:system1_gen_case2} has at most $q^{k_{s+1}+\ldots+(k_t+1)-1}$ solutions in $\F_{q^{n_{s+1}}}\times \ldots \times  \F_{q^{n_{t}}}$. For each of these solutions $(\bar{x}_{s+1},\ldots,\bar{x}_t)$, the second equation becomes
\begin{equation}\label{eq:new_subst}
    x_1+\sum_{r=2}^sa_r(x_s)=-\sum_{r=s+1}^ta_r(\bar{x}_r).
\end{equation}
Note that \eqref{eq:new_subst} is an affine equation, so its number of solution in $\F_{q^{n_1}}\times \ldots\times\F_{q^{n_s}}$
is either $0$ or equal to the cardinality of
$$T':=\left\{(x_1,\ldots,x_s) \in \F_{q^{n_1}}\times \ldots\times \F_{q^{n_s}} \colon x_1+\sum_{r=2}^sa_{r}(x_r)=0 \right\}.$$
By inductive hypothesis, 
$$|T'|\le q^{1+k_2+\ldots+k_s-1},$$
hence the total number of solutions of \eqref{eq:system1_gen_case2} is at most 
$$q^{k_{s+1}+\ldots+(k_t+1)-1}q^{1+k_2+\ldots+k_s-1}=q^{k_1+k_2+\ldots+k_t-1}.$$
 {This concludes the base case $k_1 = 1$.}
Let us assume that \eqref{eq:card_HcapS_gen} holds for  every $k_{1}'<k_1$ and every $k_{2}',\ldots,k_t'\in \mathbb N$.
Observe that, if $a_{1,k_1}=0$, then  {$a_1(x_1)=b_1(x_1)$}, where 
$b_1=(a_{1,1},\ldots,a_{1,k_1-1})$, and
$$T=\left\{(x_1,\ldots,x_t)\in \F_{q^{n_1}}\times\ldots\times  \F_{q^{n_t}} \colon b_1(x_1)+\sum_{r=2}^ta_r(x_r)=0 \right\}.$$
By inductive hypothesis with $k_1'=k_1-1$, we have $|T|\le q^{k_1'+k_2+\ldots+k_t-1}<q^{k_1+k_2+\ldots+k_t-1}$. Therefore, we may assume $a_{1,k_1}\neq 0$, and, up to rescaling, we may take $a_{1,k_1}=1$. Again, define $b_1=(a_{1,1},\ldots,a_{1,k_1-1})$, and observe in this case that 
$$a_1(x_1)=x_1^{q^{k_1-1}}+b_1(x_1).$$
Since $\gcd(n_i,n_j)=1$ for every $i\neq j$, there exists $h\in \mathbb N$ such that $h\equiv 1 \pmod{n_t}$ and $h\equiv 0 \pmod{n_i}$  for each $i \in [t-1]$. Therefore, if $(x_1,\ldots,x_t)\in T$, then
 {$$\begin{cases} -x_1^{q^{k_1-1}}-b_1(x_1)=\displaystyle\sum\limits_{r=2}^ta_r(x_r),\\
-x_1^{q^{k_1-1}}-b_1^{q^h}(x_1)=\displaystyle\sum\limits_{r=2}^{t-1}a_{r}^{q^{h}}(x_r)+a_{t}^{q^{h}}(x_t^q),
\end{cases}$$}
where the second identity follows from raising the first identity to the ${q^{h}}$-th power. By subtracting the two identities, we get the equivalent system
\begin{equation}\label{eq:system2_gen}\begin{cases} \displaystyle\sum\limits_{r=1}^ta_r(x_r)=0,\\
\displaystyle\sum\limits_{r=1}^{t}d_{r}(x_r)=0,
\end{cases}\end{equation}
where $d_{r}=a_{r}-a_{{r}}^{q^{h}}$ for each $r \in \{2,\ldots,t-1\}$, $d_1=b_1-b_1^{q^h}$, and $d_t$ is given by
$d_{t,i}=a_{t,i}-a_{{t,i-1}}^{q^{h}}$ for each $i\in [k_t+1]$ with $a_{t,k_t+1}=a_{t,0}=0$. As in the base step of the induction, one can verify that $d_t=(d_{t,1},\ldots,d_{t,k_t+1})\neq 0$. We now proceed as in the base case of the induction on $k_1$. Let $\mathcal J:=\{i \in [t-1]\,:\, d_i=0\}$. Again, we divide the proof in cases:

\noindent \underline{\textbf{Case I:}} $\mathcal J=\emptyset$. Then, all the vectors $d_i$ are nonzero for each $i \in [t]$, and 
$$T\subseteq T':=\left\{\left(x_1,x_2,\ldots,x_t\right)\in \F_{q^{n_1}}\times\ldots\times  \F_{q^{n_t}} \colon \sum_{r=1}^td_r(x_r)=0\right\}.$$
Note that $d_1\in \Fm^{k_1-1}$, while $d_i \in \Fm^{k_i}$ for each $i \in \{2,\ldots,t-1\}$ and $d_t \in \Fm^{k_t+1}$. Hence, by induction hypothesis on $k_1'=k_1-1$, we have $|T'|\le q^{k_1'+\ldots+(k_t+1)-1}=q^{k_1+\ldots+k_t-1}$.

\noindent \underline{\textbf{Case II:}} $\mathcal J\neq \emptyset$. Up to reordering the variables, we may assume $\mathcal J=[s]$, for some $s<t$ (since $d_t \neq 0$). Then, \eqref{eq:system2_gen} reads as 
\begin{equation}\label{eq:system1_more_gen_case2}\begin{cases}
\displaystyle\sum\limits_{r=s+1}^{t}d_{r}(x_r)=0, \\
\displaystyle\sum\limits_{r=1}^td_{r}(x_r)=0,
\end{cases}\end{equation}
By inductive hypothesis, the first equation in \eqref{eq:system1_more_gen_case2} has at most $q^{k_{s+1}+\ldots+(k_t+1)-1}$ solutions in $\F_{q^{n_{s+1}}}\times \ldots \times  \F_{q^{n_{t}}}$. For each of these solutions $(\bar{x}_{s+1},\ldots,\bar{x}_t)$, the second equation becomes
\begin{equation}\label{eq:new_more_subst}
    \sum_{r=1}^sd_r(x_r)=-\sum_{r=s+1}^td_r(\bar{x}_r).
\end{equation}
Note that \eqref{eq:new_more_subst} is an affine equation, so its number of solution in $\F_{q^{n_1}}\times \ldots\times\F_{q^{n_s}}$
is either $0$ or equal to the cardinality of
$$T':=\left\{(x_1,\ldots,x_s) \in \F_{q^{n_1}}\times \ldots\times \F_{q^{n_s}} \colon \sum_{r=1}^sd_{r}(x_r)=0 \right\}.$$
By inductive hypothesis on $s<t$, 
 {$$|T'|\le q^{k_1+\ldots+k_s-1},$$}
hence the total number of solutions of \eqref{eq:system1_more_gen_case2} is at most 
$$q^{k_{s+1}+\ldots+(k_t+1)-1}q^{k_1+k_2+\ldots+k_s-1}=q^{k_1+k_2+\ldots+k_t-1},$$
{ concluding  the proof.}
\end{proof}

\begin{example}
    Let $t=3$ and $n_1=5$, $n_2=3$ and $n_3=2$, and consider the direct sum
    $$\mM=\mU_{2,5}(q)\oplus\mU_{2,3}(q)\oplus \mU_{1,2}(q).$$
    for any prime power $q$. Theorem \ref{thm:direct_sum_constr_general}  states that $\mM$ is representable over $\Fm$ with  $m=n_1n_2n_3=30$, and a representation is given by $(\mS,\rho_{\mS})$ where $\mS$ is the $[10,5]_{q^{30}/q}$ system 
    $$\mS=\left\{(a,a^q,b,b^q,c) \,:\, a \in \F_{q^5}, b \in \F_{q^3}, c \in \F_{q^2} \right\}.$$
\end{example}

Thanks to Theorem \ref{thm:direct_sum_constr_general}, we can  give a positive answer to the question on the representability of the direct sum of $t$ uniform $q$-matroids.

\begin{theorem}\label{thm:representability_direct_sum} 
    Let $k_1,\ldots, k_t, n_1, \ldots n_t$ be positive integers such that $k_i< n_i$ for each $i \in [t]$. Then, the direct sum
    $$\mU_{k_1,n_1}(q) \oplus \ldots \oplus \mU_{k_t,n_t}(q)$$
    is representable. In particular, it is representable over every field $\Fm$ such that $m=m_1\cdot\ldots\cdot m_t$ for any $m_1,\ldots,m_t$ satisfying $\gcd(m_i,m_j)=1$ for each $i \neq j$ and $m_i\geq n_i$ for $i\in [t]$.
\end{theorem}

\begin{proof}
    Let $m_1,\ldots,m_t$ be positive integers as in the statement, that is $m_i\ge n_i$ for each $i \in [t]$ and $\gcd(m_i,m_j)=1$ for each $i \neq j$, and let $m=m_1\cdot\ldots\cdot m_t$. Denote by $\mathbf{k}=(k_1,\ldots,k_t)$ and by $k=k_1+\ldots+k_t$. Using Theorem \ref{thm:direct_sum_constr_general}, we can construct 
  $$        \mT_i = \{(x,x^q,\ldots,x^{q^{k_i-1}}) \colon x \in \F_{q^{m_i}}\}, \qquad \mbox{ for each } i \in [t],$$ 
    such that $\mT:=\mT_1\oplus\ldots \oplus\mT_t$ is an $[m_1+\ldots+m_t,k]_{q^m/q}$ system which is $(\Lambda_{k-1,\mathbf{k}},k-1)$-evasive. Then, by choosing any $\Fq$-subspace $V_i \subseteq \F_{q^{m_i}}$, with $\dim_{\Fq}(V_i)=n_i$ for $i \in [t]$, and defining 
      $$
        \mS_i = \{(x,x^q,\ldots,x^{q^{k_1-1}}) \colon x \in V_i\}, \qquad \mbox{ for each }i \in [t],$$
    we obtain that $\mS:=\mS_1\oplus\ldots \oplus\mS_t$ is an $\Fmk$ system which is a $(\Lambda_{k-1,\mathbf{k}},k-1)$-evasive, with $n=n_1+\ldots+n_t$. Using Theorem \ref{thm:charact_independence} we conclude.
\end{proof}

\begin{example}
    Let $t=3$ and $n_1=n_2=n_3=4$, and consider the direct sum
    $$\mM=\mU_{2,4}(q)\oplus\mU_{2,4}(q)\oplus \mU_{2,4}(q).$$
    for any prime power $q$. We consider three extension degrees $m_1,m_2,m_3$ such that $m_i\ge n_i$ for $i \in [3]$ and $\gcd(m_i,m_j)=1$ for each $1\le i<j\le 3$. A possible choice is $m_1=4$, $m_2=5$, $m_3=7$. Theorem \ref{thm:representability_direct_sum} states that 
    $\mM$ is representable over $\Fm$ with  $m=m_1m_2m_3=140$. In particular, following its proof, we can also construct the $[12,6]_{q^{140}/q}$ system $\mS$ such that $\mM\cong (\mS,\rho_{\mS})$. It is enough to fix three $4$-dimensional $\Fq$-subspaces $V_1= \F_{q^4}$, $V_2\subseteq \F_{q^5}$, $V_3\subseteq \F_{q^7}$, and take
    $$\mS:=\left\{ (a_1,a_1^q,a_2,a_2^q,a_3,a_3^q) \,:\, a_i \in V_i, i \in [3] \right\}.$$
\end{example}

    Observe that although Theorem \ref{thm:representability_direct_sum} answers affirmatively to the question about the representability of the direct sum of any $t$ uniform $q$-matroids, it does not characterize the integers $m$ for which this direct sum is $\Fm$-representable. In Section \ref{sec:point_point}, we will provide a  deeper analysis of those values in the special case of the direct sum of two uniform $q$-matroids of rank~$1$.

\section{The particular case of uniform \emph{q}-matroids of rank 1}\label{sec:point_point}

We have seen in Theorem \ref{thm:representability_direct_sum} that the direct sum of uniform $q$-matroids is always representable. However, such a result does not characterize over which extension field $\F_{q^m}$ we can have a representation. In general, this is a difficult question. In this section, we will focus on the direct sum of $t$ many uniform $q$-matroids of rank $1$.  We will first give some necessary conditions on $m$ and then we will give some sufficient conditions on $m$ ensuring that the direct sum of \emph{two} uniform $q$-matroids of rank $1$ is representable  {over $\mathbb{F}_{q^m}$}.

\subsection{Necessary conditions for the representability}

In this subsection, we will  study the extension fields over which the direct sum of $t$ many uniform $q$-matroids of rank $1$ is representable. As we have seen before, this corresponds to the study of $(\Lambda_{t-1,(1,\ldots,1)},t-1)$-evasive subspaces. By Lemma \ref{lem:boxpropgen}, these subspaces belong to the family of $(\Lambda_{1,(1,\ldots,1)},1)$-evasive subspaces. This class of subspaces have been investigated in \cite{napolitano2022linear} and in \cite{zullo2023multi} in terms of linear sets with complementary weights. A first result gives us some necessary conditions for the existence of $(\Lambda_{t-1,(1,\ldots,1)},t-1)$-evasive subspaces; this can be derived by {\cite[Theorem 4.4]{napolitano2022linear}} and \cite[Theorem 6.4]{zullo2023multi}.

\begin{theorem}\label{thm:necessary_two_fat_points}
    For $i\in [t]$, let $n_i\geq 2$, and let $\mS_i$ be an $[n_i,1]_{q^m/q}$ system. Assume that 
    $\mS_1\oplus\ldots\oplus\mS_t$ is $(\Lambda_{1,(1,\ldots,1)},1)$-evasive. Then, $n_i\le \frac{m}{2}$ for each $i \in [t]$.
\end{theorem}

Therefore, we obtain the following necessary conditions for the representability of the direct sum of uniform matroids of rank $1$.

\begin{corollary}\label{cor:necessary_two_rank1}
    Let $n_1,\ldots,n_t\geq 2$ be positive integers. If the direct sum
    $$\mU_{1,n_1}(q)\oplus \ldots \oplus \mU_{1,n_t}(q)$$
    is $\Fm$-representable, then $m\geq 2\max\{n_1,\ldots,n_t\}$.
\end{corollary}
\begin{proof}
    Let $\mathcal{S}_i$ be an $\Fm$-representation for $\mathcal{U}_{1,n_i}(q)$ for any $i \in [t]$.
    By Theorem \ref{thm:charact_independence},
    $\mathcal{S}=\oplus_{i=1}^t \mathcal{S}_i$ is an $\Fm$-representation of $$\mU_{1,n_1}(q)\oplus \ldots \oplus \mU_{1,n_t}(q)$$
    if and only if it is $(\Lambda_{t-1,(1,\ldots,1)},t-1)$-evasive. By Lemma \ref{lem:boxpropgen}, $\mathcal{S}$ is also $(\Lambda_{1,(1,\ldots,1)},1)$-evasive. So, the assertion follows by Theorem \ref{thm:necessary_two_fat_points}.
\end{proof}

\begin{remark}
    The rank-metric codes associated with the direct sum of uniform $q$-matroids of rank $1$, when representable, are the direct sum of $1$-dimensional MRD codes. This latter class of codes have been studied in \cite{santonastaso2024completely}, where such codes have been called \emph{completely decomposable} rank-metric codes. 
\end{remark}

\subsection{Construction for two summands over small extension fields}

We now specialize to the study of the direct sum of two uniform $q$-matroids of rank $1$. Combining some new and old results on the existence of $(\Lambda_{1,(1,1)},1)$-evasive subspaces we are able to prove the following result.

\begin{theorem}\label{thm:summary}
    Let $n_1,n_2\geq 2$ be two positive integers. The direct sum
    $$\mU_{1,n_1}(q)\oplus \mU_{1,n_2}(q)$$
    is $\Fm$-representable if at least one of the following holds.
    \begin{enumerate}
       \item $m$ is even and $m\geq 2\max\{n_1,n_2\}$; 
       \item $m\geq n_1n_2$;
       \item $m=t_1t_2$ with $t_1\ge n_1$ and $n_2\le \frac{t_1(t_2-1)}{2}+1$;
       \item $m=t_1t_2$ with $t_1\geq n_1,n_2$ and $t_2\geq 2$;
       \item $q=p^h$, $m=p^r$, $n_1+n_2-1\le \frac{m}{2}$.
    \end{enumerate}
\end{theorem}

The reader can verify that Theorem \ref{thm:summary}(4) is covered by the combination of Theorem \ref{thm:summary}(1) and Theorem \ref{thm:summary}(3). However, since the constructions are different and the latter ones derive from existing results in the literature, we decided to include it anyway in the statement.
\medskip 

The rest of this section is dedicated to show Theorem \ref{thm:summary}. 

\medskip 

Theorem \ref{thm:charact_independence} shows that establishing the $\F_{q^m}$-representability of the direct sum $\mU_{1,n_1}(q)\oplus \mU_{1,n_2}(q)$ is equivalent to finding an $[n_1,1]_{q^m/q}$ system $\mathcal{S}_1$ and an $[n_2,1]_{q^m/q}$ system $\mathcal{S}_2$ such that $\mathcal{S}_1\oplus\mathcal{S}_2$ is a $(\Lambda_{1,(1,1)},1)$-evasive $[n_1+n_2,2]_{q^m/q}$ system.
So, in the following we will exhibit  examples of such subspaces.
To this aim, we recall the following characterization of $(\Lambda_{1,(1,1)},1)$-evasiveness.

\begin{proposition}[{\cite[Theorem 4.1]{napolitano2022linear}}]\label{prop:Sidon_pairs}
Let $A,B\subseteq \Fm$ be two $\Fq$-subspaces with $\dim A=n_1$, $\dim B=n_2$. Then, $A\oplus B$ is $(\Lambda_{1,(1,1)},1)$-evasive
if and only if for each choice of $a_1,a_2\in A\setminus\{0\}$, $b_1,b_2\in B\setminus\{0\}$ satisfying $a_1b_2=a_2b_1$,  it must hold that $a_1=\lambda a_2$, $b_1=\lambda b_2$ for some $\lambda \in \Fq$. 
Equivalently, $A\oplus B$ is $(\Lambda_{1,(1,1)},1)$-evasive
if and only if for every $a_1,a_2\in A\setminus\{0\}$, $b_1,b_2\in B\setminus\{0\}$ for which there exists $\lambda \in \Fm$ such that $(a_1,a_2)=\lambda(b_1,b_2)$, we have that  $\lambda \in \F_q$.
\end{proposition}

Using Proposition  \ref{prop:Sidon_pairs} we can show a first example of $(\Lambda_{1,(1,1)},1)$-evasive subspace.

\begin{proposition}\label{prop:m_ge_n1n2}
    Let $n_1,n_2\geq 2$ be two positive integers. Assume that $m\geq n_1n_2$, and let $\gamma \in \Fm$ be such that $\Fm=\Fq(\gamma)$. 
Define
\begin{align*}
    \mS_1&=\{f(\gamma) \colon f \in \Fq[x]_{<n_1} \}, \\
    \mS_2 &=\{g(\gamma^{n_1}) \colon g \in \Fq[x]_{<n_2}\}.
\end{align*}
Then, $\mS_1\oplus\mS_2$ is a $(\Lambda_{1,(1,1)},1)$-evasive $[n_1+n_2,2]_{q^m/q}$ system.
\end{proposition}
\begin{proof}
By Proposition \ref{prop:Sidon_pairs}, we need to show that if $(f_1(\gamma),g_1(\gamma^{n_1}))=\lambda(f_2(\gamma),g_2(\gamma^{n_1}))$, for some nonzero $f_1,f_2\in \Fq[x]_{<n_1},g_1,g_2\in \Fq[x]_{<n_2}$,  then $\lambda\in \Fq$. Equivalently, by assuming $g_1,g_2$ being monic, our aim is to show that $\lambda=1$, that is $(f_1(\gamma),g_1(\gamma^{n_1}))=(f_2(\gamma),g_2(\gamma^{n_1}))$.

    Assume that $(f_1(\gamma),g_1(\gamma^{n_1}))=\lambda(f_2(\gamma),g_2(\gamma^{n_1}))$, then
    $$\frac{f_1(\gamma)}{f_2(\gamma)}=\lambda=\frac{g_1(\gamma^{n_1})}{g_2(\gamma^{n_1})},$$
    which in turn implies
$$f_1(\gamma)g_2(\gamma^{n_1})-f_2(\gamma)g_1(\gamma^{n_1})=0.$$
For $i \in \{1,2\},$ let us write
$$f_i(x)=\sum_{j=0}^{n_1-1}f_{i,j}x^j, \qquad g_i(x)=\sum_{j=0}^{n_2-1}g_{i,j}x^j.$$
Looking at $h(\gamma):=f_1(\gamma)g_2(\gamma^{n_1})-f_2(\gamma)g_1(\gamma^{n_1})$ as a polynomial in $\gamma$ and fixing any $i_1,i_2$ with $0\le i_2<n_1$, 
the coefficient of $\gamma^{i_1n_1+i_2}$ in $h(\gamma)$ is $f_{1,i_2}g_{2,i_1}-f_{2,i_2}g_{1,i_1}$.

In particular, since $\deg h\leq \max\{n_1\deg g_1 +\deg f_2,n_1\deg g_2+\deg f_1\}\leq n_1(n_2-1)+n_1-1<n_1n_2\leq m$, $h(\gamma)=0$ implies that all the coefficients of the polynomial $h(x)$ are $0$. Thus, for each $i_1\in \{0,\ldots,n_1\}$, $i_2\in \{0,\ldots,n_2\}$ we must have $f_{1,i_2}g_{2,i_1}-f_{2,i_2}g_{1,i_1}=0$. Together with the assumption of $g_1,g_2$ being monic, it readily follows that this implies $f_1=f_2$ and $g_1=g_2$, from which we conclude. 
\end{proof}

We can make use of another additional strategy to construct such $(\Lambda_{1,(1,1)},1)$-evasive spaces in $\Fm^2$ which, under certain assumptions, allows us to reduce the value of $m$. This is inspired by a method given in \cite{napolitano2022linear}, weakening its hypotheses.
Denote by $U\cdot V=\langle ab \colon a\in U,b \in V\rangle_{\Fq}$, for any two $\Fq$-subspaces $U$ and $V$ of $\Fm$.

\begin{lemma}\label{lem:Vq=V}
    Let $U,V\subseteq \Fm$ be two $\Fq$-subspaces and let $\xi\in \Fm^*$ be such that
    \begin{enumerate}
        \item $V^q=V$,
        \item $U\cdot V\cap \xi (U\cdot V)=\{0\}$.
    \end{enumerate}
    Then, $\mS:=U\oplus \mS_2$ with $\mS_2=\{b+\xi b^q \colon b \in V\}$ is  a $(\Lambda_{1,(1,1)},1)$-evasive $[n_1+n_2,2]_{q^m/q}$ system.
\end{lemma}

\begin{proof}
    Consider $(a,b+\xi b^q) \in \mS$ with $a,b \ne 0$. Our aim is to show that for any $\lambda \in \Fm$, $c\in U$ and $d \in V$ such that 
    \[ (a,b+\xi b^q)=\lambda(c,d+\xi d^q), \]
    we have $\lambda \in \Fq$. 
    By the above equality we get
    \[
    \left\{
    \begin{array}{ll}
    a=\lambda c,\\
    b+\xi b^q=\lambda(d+\xi d^q).
    \end{array}
    \right.
    \]
    Since also $c,d$ must be nonzero, the above system gives the following identity
    \[ c(b+\xi b^q)=a(d+\xi d^q), \]
    that is $\xi (cb^q-ad^q)=ad-cb$. Since $cb^q-ad^q, ad-cb \in U\cdot V$, we have that 
    \[ cb^q-ad^q=ad-cb=0, \]
    from which we obtain $\lambda \in \Fq$. We  conclude by using Proposition \ref{prop:Sidon_pairs}.
\end{proof}

As a consequence, from Lemma \ref{lem:Vq=V} we recover the construction in  {\cite[Corollary 4.7]{napolitano2022linear}}.

\begin{corollary}\label{cor:1}
    Let $r$ be a proper divisor of $m$ and let $\xi \in \Fm\setminus \F_{q^r}$. Let
    $\mS_1:=\F_{q^r}$ and $\mS_2:=\{a+\xi a^q \colon a \in \F_{q^r}\}$. Then, $\mS_1\oplus \mS_2$ is a $(\Lambda_{1,(1,1)},1)$-evasive $[2r,2]_{q^m/q}$ system.
\end{corollary}

We can do better in some cases, when $m$ is a prime power, using a particular family of subspaces studied in \cite{neri2023proof}. We do this by combining the following two technical lemmas with the previous Lemma \ref{lem:Vq=V}.
Assume that we have an extension field $\Fm$ of $\Fq$, where $m$ is a power of the characteristic $p$, that is, $q=p^h$ and $m=p^r$. In this case, for each $i \in \{0,\ldots,p^r\}$, define the linearized polynomial
\begin{equation}\label{eq:pi}p_i(x):=\underbrace{(x^q-x)\circ\ldots \circ (x^q-x)}_{i \text{ times}}=\sum_{j=0}^i\binom{i}{j}(-1)^{i-j}x^{q^j},\end{equation}
and the $\Fq$-subspace of $\Fm$
$$\mathcal F_i:=\ker p_i.$$

We can compute the dimension of the subspaces $\mathcal{F}_i$ and establish some of their multiplicative properties.

\begin{lemma}\label{lem:modular}
    Let $q=p^h$ and $m=p^r$, and for each $i\in \{0,\ldots,p^r\}$, let  {$\mathcal F_i:=\ker p_i\subseteq \mathbb{F}_{q^m}$}, where $p_i$ is defined as in \eqref{eq:pi}. Then:
    \begin{enumerate}
        \item $\dim_{\Fq}(\mathcal F_i)=i$, for each $i\in \{0,\ldots,p^r\}$.
        \item $\mathcal F_i^q=\mathcal F_i$, for each $i\in \{0,\ldots,p^r\}$.
        \item $\mathcal F_i\cdot \mathcal F_j\subseteq \mathcal F_{i+j-1}$ for each $i,j$ such that  {$i+j\le m+1$}.
    \end{enumerate}
\end{lemma}

\begin{proof}
    \begin{enumerate}
        \item This was proved in \cite[Lemma 3.1]{neri2023proof}.
        \item Since all the coefficients of $p_i$ are in $\Fq$, then one gets that $p_i(\alpha)=0$ if and only if $p_i(\alpha^q)=0$.
        \item This was shown in \cite[Proposition 3.3]{neri2023proof}.
    \end{enumerate}
\end{proof}

In the next result, we show that the existence of $\xi$ in Lemma \ref{lem:Vq=V} is always guaranteed, when the dimension of the subspace involved is at most $m/2$.

\begin{lemma}\label{lem:VxV}
    Let $V$ be an $\Fq$-subspace of $\Fm$ of dimension $\ell\leq \frac{m}{2}$. Then, there exists $\xi \in \Fm^*$ such that $V\cap\xi V=\{0\}$.
\end{lemma}

\begin{proof}
    Assume by contradiction that $V\cap \xi V\neq\{0\}$ for every $\xi \in \Fm^*$. Then, we must have
    $$\sum_{\xi \in \Fm^*}|(V\cap \xi V)\setminus\{0\}|\geq (q-1)(q^m-1).$$
    On the other hand, an element $\beta\in V\setminus\{0\}$ belongs to $\xi V\setminus\{0\}$ if and only if $\xi =\beta \gamma^{-1}$ for some $\gamma \in V\setminus\{0\}$. Thus, each element in $V\setminus\{0\}$ is counted exactly $|V|-1=q^\ell-1$ times  {in the sum $\sum_{\xi \in \Fm^*}|(V\cap \xi V)\setminus\{0\}|$}, and hence
    $$\sum_{\xi \in \Fm^*}|(V\cap \xi V)\setminus\{0\}|=(q^\ell-1)^2.$$
    This implies that
    $$(q^\ell-1)^2\geq (q-1)(q^m-1),$$
    from which we can derive $2\ell\ge m+1$, contradicting the hypothesis.
\end{proof}

Therefore, as a consequence we have the following construction of $(\Lambda_{1,(1,1)},1)$-evasive subspaces.

\begin{corollary}\label{cor:modular}
    Let $p$ be the characteristic of $\F_{q^m}$, $m=p^r$, and let $n_1,n_2$ be any positive integers  such that $n_1+n_2-1\le \frac{m}{2}$. Furthermore, let $\xi \in \Fm^*$ be such that $\mathcal F_{n_1+n_2-1}\cap \xi\mathcal F_{n_1+n_2-1}=\{0\}$.
    Define
    $\mS_1:=\mathcal F_{n_1}$ and $\mS_2:=\{a+\xi a^q \colon a \in \mathcal F_{n_2}\}$. Then, $\mS_1\oplus \mS_2$ is a $(\Lambda_{1,(1,1)},1)$-evasive $[n_1+n_2,2]_{q^m/q}$ system.
\end{corollary}

\begin{proof}
    Let us take $U=\mF_{n_1}$ and $V=\mF_{n_2}$. By Lemma \ref{lem:modular}(3), we have that $U\cdot V\subseteq \mF_{n_1+n_2-1}$ and $\dim(\mF_{n_1+n_2-1})=n_1+n_2-1\le \frac{m}{2}$. Hence,  Lemma \ref{lem:VxV} ensures the existence of $\xi \in \Fm^*$ such that $\mathcal F_{n_1+n_2-1}\cap \xi\mathcal F_{n_1+n_2-1}=\{0\}$. In particular, due also to Lemma \ref{lem:modular}(2), $U,V$ and $\xi$ satisfy the hypotheses of Lemma \ref{lem:Vq=V}, and thus we conclude.
\end{proof}

Combining the results of this section we are able to prove Theorem \ref{thm:summary}.

\begin{proof}[Proof of Theorem \ref{thm:summary}]
 {By Theorem \ref{thm:charact_independence}, the direct sum $\mathcal{U}_{1,n_1}(q) \oplus \mathcal{U}_{1,n_2}(q)$ is $\Fm$-representable if and only if there exist $\Fq$-subspaces $\mathcal{S}_1$ and $\mathcal{S}_2$ of $\Fm$ such that $\mathcal{S}_1 \oplus \mathcal{S}_2$ is $(\Lambda_{1,(1,1)},1)$-evasive.
The existence of such subspaces is ensured:
\begin{itemize}
\item by Corollary \ref{cor:1}, under assumptions (1) and (4),
\item by Proposition \ref{prop:m_ge_n1n2}, under assumption (2),
\item by \cite[Theorem 4.5]{napolitano2022linear}, under assumption (3), and
\item by Corollary \ref{cor:modular}, under assumption (5).
\end{itemize}
}
\end{proof}

\begin{remark}
    Let $n_2=2$ and let $n_1$ be any positive integer greater than $1$. Combining Corollary \ref{cor:necessary_two_rank1} and Theorem \ref{thm:summary}(2), we obtain that $\mU_{1,n_1}(q)\oplus \mU_{1,2}(q)$ is $\Fm$-representable if and only if $m\geq 2n_1$. This generalizes the characterization result for the $\Fm$-representability of two uniform $q$-matroids of rank $k_1=k_2=1$ and height $n_1=n_2=2$ given  \cite[Proposition 3.8]{gluesing2022representability} (see Example \ref{exa:representable_U12}).
\end{remark}
\begin{remark}
    Let $n_1,n_2\geq 3$. Then, there are only finitely many degree extensions $m$ such that we do not know yet whether the $q$-matroid $\mU_{1,n_1}(q)\oplus \mU_{1,n_2}(q)$ is $\Fm$-representable. Indeed, we know that $\mU_{1,n_1}(q)\oplus\mU_{1,n_2}(q)$ is $\Fm$-representable for every $m\geq n_1n_2$ (see Theorem \ref{thm:summary}(2)) and for every even $m \geq 2\max\{n_1,n_2\}$ (see Theorem \ref{thm:summary}(1)). The only open cases are the odd $m$ such that $2\max\{n_1,n_2\}<m <n_1n_2$, for which we can also give some answers depending on the values of $n_1, n_2$ and $q$ using Theorem \ref{thm:summary}(3,4,5). 
\end{remark}

\begin{remark}
    The smallest case left out by Theorem \ref{thm:summary} is whether the direct sum of two uniform $q$-matroids of rank $1$ and height $3$ is $\F_{q^7}$-representable. Using the algebra software \textsc{magma} \cite{magma}, we found examples of $(\Lambda_{1,(1,1)},1)$-evasive subspaces  for $q \in \{2,3,4,5,7\}$. Theorem \ref{thm:charact_independence} implies that the direct sum of two copies of $\mathcal{U}_{1,3}(q)$ is $\F_{q^7}$-representable for $q \in \{2,3,4,5,7\}$.
\end{remark}

We conclude the section by illustrating Theorem \ref{thm:summary} with a concrete example.

\begin{example}
    Let $n_1=7$ and $n_2=6$. We know by Corollary \ref{cor:necessary_two_rank1} that $\mM=\mU_{1,7}(q)\oplus \mU_{1,6}(q)$ is not $\Fm$-representable for each  {$m \in [12]$}. On the other hand, $\mM$ is $\Fm$-representable for every $m \geq 42$ (Theorem \ref{thm:summary}(2)) and for  {$m \in \{14,16,18, 20,\ldots,38, 40\}$} (Theorem \ref{thm:summary}(1)). Moreover, using Theorem \ref{thm:summary}(4), we also derive that $\mM$ is $\Fm$-representable for $m \in \{21,27,33,35,39\}$.
    Finally, Theorem \ref{thm:summary}(5) implies that $\mM$ is $\Fm$-representable also when 
    \begin{equation}\label{eq:m_char}m=\begin{cases}
        25 &  \mbox{ if } q=5^h, \\
        29 & \mbox{ if }  q=29^h, \\
        31 & \mbox{ if } q=31^h, \\
        37 & \mbox{ if } q=37^h, \\
        41 & \mbox{ if } q=41^h. 
    \end{cases} \end{equation}
    The only open cases are for  {$m \in \{13,15,17,19,23,25^*,29^*,31^*,37^*,41^*\}$}, where the $^*$ indicates that the case is open except for some values of the characteristic of the field, according to \eqref{eq:m_char}.
\end{example}

\section{Conclusions and open questions}\label{sec:conclusion}

In this paper we studied the representability of the direct sum of $t$ uniform $q$-matroids. By establishing a geometric characterization of the representability, we identified the critical role of evasiveness properties and cyclic flats. Our results show that the direct sum of uniform $q$-matroids is representable if we can find $q$-systems satisfying a certain evasiveness property.

We have constructed a $q$-system over sufficiently large fields that meets these properties, affirming that the direct sum of $t$ uniform $q$-matroids is always representable. This construction not only demonstrates representability but also raises intriguing questions regarding the minimal extension field required for such representations.

Our exploration of the specific case involving the direct sum of two uniform $q$-matroids of rank $1$ yielded more detailed insights. By leveraging the existing research on linear sets with complementary weights and on recent results on rank-metric codes, we determined more precise extension field order assumptions. 

At this point, there are some natural problems that arise.

\begin{itemize} 
       \item The problem of characterizing the extensions over which the direct sum of $t$ uniform $q$-matroids is representable stays widely open. We gave only some partial answers to this question in the general case. Only in the special case of $t=2$ uniform $q$-matroids of rank $1$, we were able to provide more accurate results, leaving only finitely many cases unsolved. However, the problem is still open, even in the case of $t$ uniform $q$-matroids of rank $1$.
    \item In order to study the problem of the representability of the direct sum of (not necessarily uniform) $q$-matroids  we would need a generalization of Theorem \ref{thm:charact_independence}. This may allow us to use similar arguments to characterize the representability of $q$-matroids. 
\end{itemize}

\section*{Acknowledgments}
The authors are thankful to Giovanni Longobardi, Rocco Trombetti and Lei Xu for pointing out a mistake in the proof $(2)\Rightarrow (3)$ of Theorem \ref{thm:charact_independence} in a previous version of the paper. Moreover, the authors are thankful to the anonymous referees for their valuable comments and dedication.
Gianira~N. Alfarano was supported by the Swiss National Foundation through grant no. 210966, by the FWO (Research Foundation Flanders) grant no. 1273624N and by the ANR through grant no ANR-24-CPJ1-0075-01. Alessandro Neri was supported by   the FWO (Research Foundation Flanders) grant no. 12ZZB23N and by the INdAM - GNSAGA Project CUP E53C24001950001  ``Noncommutative polynomials in coding theory''. Ferdinando Zullo is very grateful for the hospitality of Ghent University, where he was a visiting researcher during the development of this research in September 2023. He is supported by the project COMBINE  of the University of Campania ``Luigi Vanvitelli'' and by Erasmus+ UE (Programme Countries) from Unicampania and by the INdAM - GNSAGA.

\bibliographystyle{abbrv}
\bibliography{references}

\begin{thebibliography}{10}

\bibitem{adriaensen2023minimum}
S.~Adriaensen and P.~Santonastaso.
\newblock On the minimum size of linear sets.
\newblock {\em The Electronic Journal of Combinatorics}, 31(4), Nov. 2024.

\bibitem{alfarano2022linear}
G.~N. Alfarano, M.~Borello, A.~Neri, and A.~Ravagnani.
\newblock Linear cutting blocking sets and minimal codes in the rank metric.
\newblock {\em Journal of Combinatorial Theory, Series A}, 192:105658, Nov. 2022.

\bibitem{alfarano2022cyclic}
G.~N. Alfarano and E.~Byrne.
\newblock The cyclic flats of a $q$-matroid.
\newblock {\em J. Algebr. Comb.}, pages 1--30, 2024.

\bibitem{bartoli2021evasive}
D.~Bartoli, B.~Csajbók, G.~Marino, and R.~Trombetti.
\newblock Evasive subspaces.
\newblock {\em Journal of Combinatorial Designs}, 29(8):533–551, May 2021.

\bibitem{bartoli2022exceptional}
D.~Bartoli, G.~Marino, A.~Neri, and L.~Vicino.
\newblock Exceptional scattered sequences.
\newblock {\em Algebraic Combinatorics}, 7(5):1405–1431, Oct. 2024.

\bibitem{blokhuis2000scattered}
A.~Blokhuis and M.~Lavrauw.
\newblock Scattered spaces with respect to a spread in {${\rm PG}(n,q)$}.
\newblock {\em Geom. Dedicata}, 81(1-3):231--243, 2000.

\bibitem{magma}
W.~BOSMA, J.~CANNON, and C.~PLAYOUST.
\newblock The {M}agma algebra system {I}: The user language.
\newblock {\em Journal of Symbolic Computation}, 24(3–4):235–265, Sept. 1997.

\bibitem{byrne2021weighted}
E.~Byrne, M.~Ceria, S.~Ionica, and R.~Jurrius.
\newblock Weighted subspace designs from $q$-polymatroids.
\newblock {\em J. Comb. Theory Ser. A}, 201:105799, 2024.

\bibitem{byrne2022constructions}
E.~Byrne, M.~Ceria, and R.~Jurrius.
\newblock Constructions of new {$q$}-cryptomorphisms.
\newblock {\em J. Combin. Theory Ser. B}, 153:149--194, 2022.

\bibitem{ceria2021direct}
M.~Ceria and R.~Jurrius.
\newblock The direct sum of q-matroids.
\newblock {\em Journal of Algebraic Combinatorics}, 59(2):291–330, Feb. 2024.

\bibitem{crapo1964theory}
H.~Crapo.
\newblock {\em On the theory of combinatorial independence}.
\newblock PhD thesis, Massachusetts Institute of Technology, Department of Mathematics, 1964.

\bibitem{csajbok2021generalising}
B.~Csajbók, G.~Marino, O.~Polverino, and F.~Zullo.
\newblock Generalising the scattered property of subspaces.
\newblock {\em Combinatorica}, 41(2):237–262, Feb. 2021.

\bibitem{de78}
P.~Delsarte.
\newblock Bilinear forms over a finite field, with applications to coding theory.
\newblock {\em J. Comb. Theory Ser. A}, 25(3):226--241, 1978.

\bibitem{ga85a}
E.~M. Gabidulin.
\newblock Theory of codes with maximum rank distance.
\newblock {\em Problemy Peredachi Informatsii}, 21(1):3--16, 1985.

\bibitem{ghorpade2020polymatroid}
S.~R. Ghorpade and T.~Johnsen.
\newblock A polymatroid approach to generalized weights of rank metric codes.
\newblock {\em Des. Codes, Cryptogr.}, 88(12):2531--2546, 2020.

\bibitem{gluesing2021q}
H.~Gluesing-Luerssen and B.~Jany.
\newblock $q$-polymatroids and their relation to rank-metric codes.
\newblock {\em J. Algebr. Comb.}, pages 1--29, 2022.

\bibitem{gluesing2023coproducts}
H.~Gluesing-Luerssen and B.~Jany.
\newblock Coproducts in categories of $q$-matroids.
\newblock {\em Eur. J. Comb.}, 112:103733, 2023.

\bibitem{gluesing2023decompositions}
H.~Gluesing-Luerssen and B.~Jany.
\newblock Decompositions of $q$-matroids using cyclic flats.
\newblock {\em SIAM Journal on Discrete Mathematics}, 38(4):2940–2970, Dec. 2024.

\bibitem{gluesing2022representability}
H.~Gluesing-Luerssen and B.~Jany.
\newblock Representability of the direct sum of $q$-matroids.
\newblock {\em Journal of Algebraic Combinatorics}, 61(4), June 2025.

\bibitem{gorla2019rank}
E.~Gorla, R.~Jurrius, H.~H. L{\'o}pez, and A.~Ravagnani.
\newblock Rank-metric codes and $q$-polymatroids.
\newblock {\em J. Algebr. Comb.}, 52:1--19, 2020.

\bibitem{gruica2022generalised}
A.~Gruica, A.~Ravagnani, J.~Sheekey, and F.~Zullo.
\newblock Generalised evasive subspaces.
\newblock {\em J. Comb. Des.}, 2022.

\bibitem{guralnick1994invertible}
R.~M. Guralnick.
\newblock Invertible preservers and algebraic groups.
\newblock {\em Linear Algebra and its Applications}, 212–213:249–257, Nov. 1994.

\bibitem{jurrius2018defining}
R.~Jurrius and G.~Pellikaan.
\newblock Defining the $q$-analogue of a matroid.
\newblock {\em Electron. J. Comb.}, 25(3), 2018.

\bibitem{marino2023evasive}
G.~Marino, A.~Neri, and R.~Trombetti.
\newblock Evasive subspaces, generalized rank weights and near {MRD} codes.
\newblock {\em Discrete Mathematics}, 346(12):113605, Dec. 2023.

\bibitem{napolitano2022linear}
V.~Napolitano, O.~Polverino, P.~Santonastaso, and F.~Zullo.
\newblock Linear sets on the projective line with complementary weights.
\newblock {\em Discrete Mathematics}, 345(7):112890, July 2022.

\bibitem{neri2021geometry}
A.~Neri, P.~Santonastaso, and F.~Zullo.
\newblock The geometry of one-weight codes in the sum-rank metric.
\newblock {\em Journal of Combinatorial Theory, Series A}, 194:105703, Feb. 2023.

\bibitem{neri2023proof}
A.~Neri and M.~Stanojkovski.
\newblock A proof of the {Etzion-Silberstein} conjecture for monotone and {MDS}-constructible {F}errers diagrams.
\newblock {\em Journal of Combinatorial Theory, Series A}, 208:105937, Nov. 2024.

\bibitem{polverino2010linear}
O.~Polverino.
\newblock Linear sets in finite projective spaces.
\newblock {\em Discrete Math.}, 310(22):3096--3107, 2010.

\bibitem{randrianarisoa2020geometric}
T.~H. Randrianarisoa.
\newblock A geometric approach to rank metric codes and a classification of constant weight codes.
\newblock {\em Designs, Codes and Cryptography}, 88(7):1331–1348, Mar. 2020.

\bibitem{roth1996tensor}
R.~Roth.
\newblock Tensor codes for the rank metric.
\newblock {\em Proceedings of 1995 IEEE International Symposium on Information Theory}, page 239, 1995.

\bibitem{santonastaso2024completely}
P.~Santonastaso.
\newblock Completely decomposable rank-metric codes.
\newblock {\em Linear Algebra and its Applications}, 2025.

\bibitem{sheekey2019scatterd}
J.~Sheekey.
\newblock ({S}cattered) {L}inear {S}ets are to {R}ank-{M}etric {C}odes as {A}rcs are to {H}amming-{M}etric {C}odes.
\newblock In M.~Greferath, C.~Hollanti, and J.~Rosenthal, editors, {\em Oberwolfach Report No. 13/2019}, 2019.

\bibitem{Sheekey2020}
J.~Sheekey and G.~Van~de Voorde.
\newblock Rank-metric codes, linear sets, and their duality.
\newblock {\em Designs, Codes and Cryptography}, 88(4):655–675, Dec. 2019.

\bibitem{shiromoto2019codes}
K.~Shiromoto.
\newblock Codes with the rank metric and matroids.
\newblock {\em Des. Codes, Cryptogr.}, 87(8):1765--1776, 2019.

\bibitem{zini2021scattered}
G.~Zini and F.~Zullo.
\newblock Scattered subspaces and related codes.
\newblock {\em Designs, Codes and Cryptography}, 89(8):1853–1873, May 2021.

\bibitem{zullo2023multi}
F.~Zullo.
\newblock Multi-orbit cyclic subspace codes and linear sets.
\newblock {\em Finite Fields and Their Applications}, 87:102153, Mar. 2023.

\end{thebibliography}

\end{document}